\pdfoutput=1
\documentclass[reqno]{shinyart}

\usepackage[utf8]{inputenc}

\usepackage{shinybib}

\usepackage{enumitem} 
\usepackage{booktabs}
\usepackage[scientific-notation=true]{siunitx}
\newtheorem{assumption}[theorem]{Assumption}
\crefname{assumption}{Assumption}{Assumptions}

\newcommand{\R}{\mathbb{R}}
\newcommand{\N}{\mathbb{N}}
\renewcommand{\div}{\operatorname{div}}

\newcommand{\LL}{\mathcal{L}}
\newcommand{\TT}{\mathcal{T}}
\newcommand{\II}{\mathcal{I}}
\newcommand{\dual}[2]{\langle #1 , #2 \rangle}
\newcommand{\weakly}{\rightharpoonup}
\renewcommand{\d}{\mathrm{d}}

\newcommand{\argmin}{\operatorname{argmin}}
\newcommand{\diag}{\operatorname{diag}}
\newcommand{\supp}{\operatorname{supp}}
\newcommand{\prox}{\operatorname{prox}}
\newcommand{\mx}{\max\nolimits}
\newcommand{\embed}{\hookrightarrow}

\title{Optimal control of a non-smooth semilinear elliptic equation}
\author{%
    Constantin Christof\thanks{%
        TU Dortmund, Faculty of Mathematics,
        Vogelpothsweg~87,
        44227 Dortmund, Germany,
        (\email{constantin.christof@tu-dortmund.de},
        \email{christian.meyer@math.tu-dortmund.de},
        \email{stephan.walther@tu-dortmund.de})
    }
    \and
    Christian Clason\thanks{%
        University of Duisburg-Essen, Faculty of Mathematics,
        Thea-Leymann-Str.~9,
        45127 Essen, Germany
    (\email{christian.clason@uni-due.de})}
    \and
    Christian Meyer\footnotemark[1]\textsuperscript{ ,}\footnote{Corresponding author}
    \and
    Stephan Walther\footnotemark[1]
}             
\date{November 27, 2017}

\hypersetup{
    pdftitle={Optimal control of a non-smooth semilinear elliptic equation},
    pdfauthor={Constantin Christof, Christian Clason, Christian Meyer, Stephan Walther},
    pdfkeywords={optimal control of PDEs, non-smooth optimization, Bouligand subdifferential, strong stationarity, semi-smooth Newton method}
}

\begin{document}

\maketitle

\begin{abstract}
    This paper is concerned with an optimal control problem governed by 
    a non-smooth semilinear elliptic equation. We show that the control-to-state mapping 
    is directionally differentiable and precisely characterize its Bouligand sub\-differential. 
    By means of a suitable regularization, first-order optimality conditions 
    including an adjoint equation are derived and afterwards interpreted in light of the 
    previously obtained characterization. In addition, the directional derivative 
    of the control-to-state mapping is used to establish strong stationarity conditions.
    While the latter conditions are shown to be stronger, we demonstrate by numerical examples that the former conditions are amenable to numerical solution using a semi-smooth Newton method.
\end{abstract}


\section{Introduction}

In this paper, we consider the following non-smooth semilinear elliptic optimal control problem
\begin{equation}\tag{P}\label{eq:p}
    \left.
        \begin{aligned}
            &\min_{u \in L^2(\Omega), y \in H_0^1(\Omega)}  \quad J(y,u)\\
            &\text{s.t.} 
            \quad -\Delta y + \max(0, y) = u \; \text{ in }\Omega,
        \end{aligned}
    \qquad \right\}
\end{equation}
where $\Omega \subset \R^d$, $d\in \N$, is a bounded domain and $J$ is a smooth objective; 
for precise assumptions on the data, we refer to \cref{assu:standing} below.
The semilinear PDE in \eqref{eq:p} models the deflection of a stretched thin membrane partially covered by water (see \cite{Kikuchi1984}); a similar equation arises in free boundary problems for a confined plasma; see, e.g., \cite{Temam:1975,Rappaz:1984,Kikuchi1984}.

The salient feature of \eqref{eq:p} is of course the occurrence of the
non-smooth $\max$-function in the equality constraint in \eqref{eq:p}. This causes the associated 
control-to-state mapping $u\mapsto y$ to be non-smooth as well, and hence standard techniques for obtaining first-order necessary optimality conditions that are based on the adjoint of the G\^ateaux-derivative of 
the control-to-state mapping cannot be applied. 
One remedy to cope with this challenge is to apply generalized differentiability concepts
for the derivation of optimality conditions.
Such concepts and the resulting conditions can be roughly grouped in two classes:
\begin{enumerate}[label=(\roman*)]
    \item (generalized) directional derivatives, leading to ``purely primal'' optimality conditions stating that the directional derivatives of the reduced objective in feasible directions are non-negative;
    \item various notions of subdifferentials, leading to abstract optimality conditions stating that zero is contained in a suitable subdifferential of the reduced objective.
\end{enumerate}
For an introductory treatment of generalized derivatives and their relation in finite and/or infinite dimensions, we refer to, e.g., the textbooks \cite{Clarke:1990,rw04,schirotzek,Penot}; a more general, infinite-dimensional, treatment can be found in \cite{Mordukhovich}.
However, with the exception of the convex setting, concrete (and, in particular, numerically tractable) characterizations of these abstract conditions are only available in a restricted set of situations; see, e.g., \cite{Outrata2005,Henrion:2010,ClasonValkonen15,Mehlitz2016}.
To be more precise, it is frequently unclear if these abstract optimality conditions 
-- whether of type (i) or (ii) -- are equivalent to optimality systems involving dual variables. (We recall that even in the convex setting, deriving such systems from conditions of type (ii) requires regularity conditions of, e.g., Slater type, which do not hold in all cases. Such regularity conditions are even more restrictive for Clarke and limiting subdifferentials.)
For selected optimal control problems governed by variational inequalities (VIs), purely primal optimality conditions of type (i) 
have been transformed into optimality systems known as strong stationarity conditions. 
We refer to \cite{m76, mp84, ojs11, W14, W17} for obstacle-type problems, 
to \cite{hmw10} for static elastoplasticity, to \cite{Wachsmuth16b} for bilevel optimal control of ODEs,  
and to \cite{dlRM} for VIs of the second kind. 
In \cite{ms16}, the equivalence of strong stationarity to purely primal conditions of type (i) 
is shown for optimal control of non-smooth semilinear parabolic equations. 
Concerning the comparison of conditions of type (ii) with optimality systems involving dual variables for optimal control of non-smooth problems,
the literature is comparatively scarce. Frequently, regularization and 
relaxation methods, respectively, are used to derive optimality systems, and there are numerous 
contributions in this field; we only refer to \cite{barbu, tiba, b98, ik00, h08, sw13} and 
the references therein. A limit analysis for vanishing regularization then yields an optimality system
for the original non-smooth problem which is usually of intermediate strength and less rigorous 
compared to strong stationarity. 
A classification of the different optimality systems involving dual variables for the case of optimal control of 
the obstacle problem can be found in \cite{hkopacka2009, HW17}. In \cite{ojs11, HMS14, Wachsmuth16a}, 
optimality conditions for the case of VIs of the first kind are obtained by using limiting normal cones. 
As shown in \cite[Thm.~5.7]{HW17}, this approach does in general not lead to 
optimality conditions which are more rigorous than what is obtained by regularization 
(even if the limiting normal cone in the spirit of Mordukhovich is used).
On the other hand, it is known for the case of finite-dimensional mathematical programs with equilibrium constraints that under suitable assumptions, the optimality conditions 
obtained via regularization are equivalent to zero being in the Clarke subdifferential of the 
reduced objective; see, e.g., \cite[Sec.~2.3.3]{cervinka}.
The latter two results show that a comparison of an optimality system 
involving dual variables with conditions of type (ii) is in general far from evident,
and we are not aware of any contributions in this direction for the case of 
optimal control of non-smooth PDEs.
This is of particular interest, however, since optimality conditions of type (ii) may be satisfied 
by accumulation points of sequences generated by optimization algorithms; 
see \cite{HS16}.

The aim of our paper is to investigate this connection for the particular optimal control problem \eqref{eq:p}.
For this purpose, we turn our attention to the Bouligand subdifferential of the control-to-state mapping, 
defined as the set of limits of Jacobians of smooth points in the spirit of, 
e.g., \cite[Def.~2.12]{okz98} or \cite[Sec.~1.3]{kk02}. Note that
in infinite-dimensional spaces, one has to pay attention to the topology underlying these limit processes 
so that multiple notions of Bouligand subdifferentials arise, see \cref{def:bouli} below.
We will precisely characterize these subdifferentials and use this result to interpret the 
optimality conditions arising in the regularization limit. 
We emphasize that the regularization and the associated limit analysis are not novel and fairly straightforward. 
The main contribution of our work is the characterization of the Bouligand subdifferential as a 
set of linear PDE solution operators; see \cref{th:endgame}.  
This characterization allows a comparison of the optimality conditions obtained by regularization with standard optimality conditions of type (ii) (specifically, involving Bouligand and Clarke subdifferentials), 
which shows that the former are surprisingly strong; cf.~\cref{theorem:summary}. 
On the other hand, it is well-known that one loses information in the regularization limit, 
and the same is observed in case of \eqref{eq:p}. In order to see this, 
we establish another optimality system, which is equivalent to a purely primal optimality condition of type (i). It will turn out that the optimality system derived in this way 
is indeed stronger than the one obtained via regularization since it contains an additional sign condition 
for the adjoint state. It is, however, not clear how to solve these strong stationarity conditions numerically. 
In contrast to this, the optimality system arising in the regularization limit allows a reformulation 
as a non-smooth equation that is amenable to solution by semi-smooth Newton methods. We emphasize that we do not employ the regularization procedure for numerical computations, but 
directly solve the limit system instead. 
Our work includes first steps into this direction, but the numerical results are preliminary 
and give rise to future research.

Let us finally emphasize that our results and the underlying analysis are in no way limited to the 
PDE in \eqref{eq:p}. Instead, the arguments can easily be adapted to more general cases involving 
a piecewise $C^1$-function rather than the $\max$-function and a (smooth) divergence-gradient-operator 
instead of the Laplacian. However, in order to keep the discussion as concise as possible and to be able to 
focus on the main arguments, we restrict the analysis to the specific PDE under consideration.

The outline of the paper is as follows:
This introduction ends with a short subsection on our notation and the standing assumptions.
We then turn to the control-to-state mapping, showing that it is globally Lipschitz and directionally differentiable
and characterizing  points where it is G\^ateaux-differentiable. 
\Cref{sec:bouli} is devoted to the characterization of the Bouligand subdifferentials. 
We first state necessary conditions that elements of the subdifferentials have to fulfill. 
Afterwards, we prove that these are also sufficient, which is by far more involved compared to 
showing their necessity. 
In \cref{sec:fon}, we first shortly address the regularization and the corresponding limit analysis. 
Then we compare the 
optimality conditions arising in the regularization limit with our findings from \cref{sec:bouli}. 
The section ends with the derivation of the strong stationarity conditions.
\Cref{sec:numerics} deals with the numerical solution of the optimality system derived via regularization. 
The paper ends with an appendix containing some technical lemmas whose proofs are difficult to find in the literature.

\subsection{Notation and standing assumptions}

Let us shortly address the notation used throughout the paper. 
In what follows, $\Omega$ always denotes a bounded domain.
By 
$\mathbb{1}_M : \mathbb{R}^d \to \{0,1\}$
we denote the characteristic function of a set $M\subset \R^d$. 
By $\lambda^d$ we denote the $d$-dimensional Lebesgue measure.
Given a (Lebesgue-)measurable function $v : \Omega \to \R$, we abbreviate the set $\{x\in \Omega: v(x) = 0\}$
by $\{v = 0\}$; the sets $\{v > 0\}$ and $\{v<0\}$ are defined analogously. 
Note that in what follows, we always work with the notion of Lebesgue measurability (e.g., when talking about $L^p$-spaces or representatives), although we could equivalently work with Borel measurability here.
As usual, the Sobolev space $H^1_0(\Omega)$ is defined as the closure of $C_c^\infty(\Omega)$ 
with respect to the $H^1$-norm. Moreover, we define the space
\begin{equation*}
    Y  := \{ y \in H^1_0(\Omega) : \Delta y \in L^2(\Omega) \}.
\end{equation*}
Equipped with the scalar product 
\begin{equation*}
    (y,v)_Y := \int_\Omega \big(\Delta y \, \Delta v + \nabla y \cdot \nabla v + y\,v\big)\, \d x,
\end{equation*}
$Y$ becomes a Hilbert space. Here and in the remainder of this paper, 
$\Delta = \div \circ \nabla : H^1_0(\Omega) \to H^{-1}(\Omega)$ denotes the distributional Laplacian. 
Note that $Y$ is compactly embedded in $H^1_0(\Omega)$ since, for any sequence 
$(y_n) \subset Y$ with $y_n \weakly y$ in $Y$, we have 
\begin{equation*}
    \|y_n - y\|_{H^1(\Omega)}^2 
    = - \int_\Omega \Delta (y_n-y) \,(y_n - y)\,\d x + \|y_n - y\|_{L^2(\Omega)}^2 \to 0
\end{equation*}
by the compact embedding of $H^1_0(\Omega)$ into $L^2(\Omega)$ (cf.~\cite[Thm.~7.22]{gilbarg2001elliptic}).
Since $Y$ is isometrically isomorphic to the subset 
\begin{equation*}
    \{(y, \omega, \delta) \in L^2(\Omega;\R^{d+2}): \exists\, v\in H^1_0(\Omega) \text{ with }
    y = v, \omega = \nabla v, \delta = \Delta v \text{ a.e. in }\Omega\}
\end{equation*}
of the separable space $L^2(\Omega;\R^{d+2})$, it is separable as well. 
Note that the solution operator $S : u \mapsto y$ associated with the PDE $-\Delta y + \max(0, y) = u$ in \eqref{eq:p} is bijective as a function from $L^2(\Omega)$ to $Y$ (cf.~\cref{prop:basic} below). The space $Y$ is thus the natural choice for the image space of the control-to-state mapping appearing in problem  \eqref{eq:p}. If the boundary $\partial \Omega$ possesses enough regularity (a $C^{1,1}$-boundary would be sufficient here), then $Y$ is isomorphic to $H_0^1(\Omega) \cap H^2(\Omega)$ by the classical regularity theory for the Laplace operator, cf.~\cite[Lem.~9.17]{gilbarg2001elliptic}.

With a little abuse of notation, 
in what follows we will
denote the Nemytskii operator 
induced by the $\max$-function (with different domains and ranges) by the same symbol. In the same way, we will denote by $\max'(y;h)$ the directional derivative of $y\mapsto \max(0,y)$ in the point $y$ in direction $h$, both considered as a scalar function and as the corresponding Nemytskii operator.

Throughout the paper, we will make the following standing assumptions.
\begin{assumption}\label{assu:standing}
    The set $\Omega\subset \R^d$, $d\in \N$, is a bounded domain.
    The objective functional $J: Y \times L^2(\Omega) \to \R$ in \eqref{eq:p}
    is weakly lower semi-continuous and continuously Fr\'echet-differentiable.
\end{assumption}
Note that we do not impose any regularity assumptions on the boundary of $\Omega$.

\section{Directional differentiability of the control-to-state mapping}
We start the discussion of the optimal control problem \eqref{eq:p} by investigating its PDE constraint, 
showing that it is uniquely solvable and that the associated solution operator is directionally differentiable.

\begin{proposition}\label{prop:basic}
    For all $u\in H^{-1}(\Omega)$, there exists a unique solution $y \in H^1_0(\Omega)$ to 
    \begin{equation}\tag{PDE}\label{eq:pde}
        -\Delta y + \max(0, y) = u.
    \end{equation}
    Moreover, the solution operator $S: u \mapsto y$ associated with \eqref{eq:pde} 
    is well-defined and globally Lipschitz continuous as a function from $L^2(\Omega)$ to $Y$.
\end{proposition}

\begin{proof}
    The arguments are standard. First of all, Browder and Minty's theorem on monotone 
    operators yields the existence of a unique solution in $H^1_0(\Omega)$. 
    If $u\in L^2(\Omega)$, then a simple bootstrapping argument implies $y \in Y$. 
    To prove the Lipschitz continuity of the solution mapping $S$, we consider two arbitrary but fixed $u_1, u_2 \in L^2(\Omega)$ with associated solutions $y_1 := S(u_1)$ and $y_2 := S(u_2)$. Using again the monotonicity, we obtain straightforwardly that $\|y_1 - y_2\|_{H^1} \leq C \|u_1 - u_2\|_{L^2}$ holds with some absolute constant $C>0$. From \eqref{eq:pde} and the global Lipschitz continuity of the $\max$-function, we now infer
    \begin{equation*}
        \|\Delta (y_1 - y_2)\|_{L^2} = \| \max(0, y_1) - \max(0, y_2) - u_1 + u_2\|_{L^2} \leq (C + 1) \|u_1 - u_2\|_{L^2}.
    \end{equation*}
    The above shows that $S$ is even globally Lipschitz as a function from $L^2(\Omega)$ to $Y$ and completes the proof. 
\end{proof}

\begin{theorem}[directional derivative of $S$]\label{thm:rabl}
    Let $u, h \in L^2(\Omega)$ be arbitrary but fixed, set $y:=S(u)\in Y$, 
    and let $\delta_h \in Y$ be the unique solution to 
    \begin{equation}\label{eq:ddpde}
        - \Delta \delta_h +  \mathbb{1}_{\{y = 0\}}\max(0, \delta_h) + \mathbb{1}_{\{y >0\}} \delta_h = h.
    \end{equation}
    Then it holds
    \begin{alignat*}{3} 
        h_n  \weakly h \;\text{ in }L^2(\Omega),\,t_n \to 0^+ & \quad &
        \Longrightarrow & \quad & \frac{S(u + t_n h_n ) - S(u)}{t_n} \weakly \delta_h \;\text{ in }Y\\
        \intertext{and} 
        h_n  \to h \;\text{ in }L^2(\Omega),\, t_n \to 0^+ & \quad & 
        \Longrightarrow & \quad & \frac{S(u + t_n h_n ) - S(u)}{t_n} \to \delta_h \;\text{ in }Y.
    \end{alignat*}
    In particular, the solution operator $S : L^2(\Omega) \to Y$ associated with 
    \eqref{eq:pde} is Hadamard directionally differentiable with $S'(u; h) = \delta_h \in Y$ in all points $u \in L^2(\Omega)$ 
    in all directions $h \in L^2(\Omega)$.
\end{theorem}

\begin{proof}
    First observe that for every $h\in L^2(\Omega)$, \eqref{eq:ddpde} admits a unique solution 
    $\delta_h \in Y$ by exactly the same arguments as in the proof of \cref{prop:basic}. 
    Note moreover that \eqref{eq:ddpde} is equivalent to
    \begin{equation*}
        - \Delta \delta_h +  \mx'(y;\delta_h) = h.
    \end{equation*}
    Now let $u, h \in L^2(\Omega)$ be arbitrary but fixed 
    and let $(t_n) \subset (0, \infty)$ and $(h_n) \subset L^2(\Omega)$ be sequences with $t_n \to 0$ and $h_n \weakly h$ in $L^2(\Omega)$. 
    We abbreviate $y_n := S(u + t_n h_n) \in Y$. Subtracting the equations for $y$ and $\delta_h$ from the one for $y_n$ yields
    \begin{equation}\label{eq:diffquotpde}
        \begin{aligned}[b]
            -\Delta \Big( \frac{y_n - y}{t_n} - \delta_h \Big) 
            & =  h_n - h + \frac{\max(0,y + t_n \delta_h) - \max(0,y_n)}{t_n} \\[-1mm]
            & \quad - \Big(\frac{\max(0,y + t_n \delta_h) - \max(0,y)}{t_n} - \mx'(y;\delta_h)\Big).
        \end{aligned}
    \end{equation}  
    Testing this equation with $(y_n - y)/t_n - \delta_h$ and using the monotonicity of the $\max$-operator,
    we obtain that there exists a constant $C>0$ independent of $n$ with
    \begin{equation*}
        \Big\| \frac{y_n - y}{t_n} - \delta_h \Big\|_{H^1(\Omega)} 
        \leq \| h_n - h\|_{H^{-1}(\Omega)} 
        + \Big\|\frac{\max(0,y + t_n \delta_h) - \max(0,y)}{t_n} - \mx'(y;\delta_h)\Big\|_{L^2(\Omega)}.
    \end{equation*}
    Now the compactness of $L^2(\Omega) \embed H^{-1}(\Omega)$ and the 
    directional differentiability of $\max: L^2(\Omega) \to L^2(\Omega)$ (which directly follows 
    from the directional differentiability of $\max: \R \to \R$ and Lebesgue's dominated 
    convergence theorem) give
    \begin{equation}\label{eq:diffconvH1}
        \frac{y_n - y}{t_n} - \delta_h \to 0 \quad \text{in } H^1_0(\Omega).
    \end{equation}
    As $\max: L^2(\Omega) \to L^2(\Omega)$ is also Lipschitz continuous and thus Hadamard-differentiable, 
    \eqref{eq:diffconvH1} implies 
    \begin{equation}\label{eq:maxhadamard}
        \frac{\max(0,y_n) - \max(0,y)}{t_n} - \mx'(y;\delta_h) \to 0 
        \quad \text{in }L^2(\Omega).
    \end{equation}
    Hence, \eqref{eq:diffquotpde} yields that the sequence $(y_n - y)/t_n - \delta_h$ is bounded in $Y$ 
    and thus (possibly after transition to a subsequence) converges weakly in $Y$. 
    Because of \eqref{eq:diffconvH1}, the weak limit is zero and therefore unique so that the whole 
    sequence converges weakly to zero. This implies the first assertion. 
    If now $h_n$ converges strongly to $h$ in $L^2(\Omega)$, then \eqref{eq:diffquotpde},
    \eqref{eq:diffconvH1} and \eqref{eq:maxhadamard}
    yield $\Delta ( (y_n - y)/t_n - \delta_h ) \to 0$ in $L^2(\Omega)$.
    From the definition of the norm $\|\cdot\|_Y$, it now readily follows that  $(y_n - y)/t_n - \delta_h \to 0$ in $Y$.
    This establishes the second claim.
\end{proof}
\Cref{thm:rabl} allows a precise characterization of points where $S$ is G\^ateaux-differentiable. 
This will be of major importance for the study of the Bouligand subdifferentials in the next section.

\begin{corollary}[characterization of G\^ateaux-differentiable points]\label{cor:GateauxCrit}
    Let $u \in L^2(\Omega)$ be arbitrary but fixed. Then the following are  equivalent:
    \begin{enumerate}[label=(\roman*)]
        \item $y = S(u)$ satisfies $\lambda^d(\{y =0\}) = 0$.
            \label{enum:Gateaux:i}
        \item $S : L^2(\Omega) \to Y$ is G\^ateaux-differentiable in $u$, i.e., $h\mapsto S'(u; h) \in  \mathcal{L}(L^2(\Omega), Y)$.\label{enum:Gateaux:ii}
        \item $S'(u; h) = -S'(u; -h)$ holds for all $h \in L^2(\Omega)$.
            \label{enum:Gateaux:iii}
    \end{enumerate}
    If one of the above holds true,
    then the directional derivative $\delta_h = S'(u; h) \in Y$ 
    in a direction $h \in L^2(\Omega)$  is uniquely characterized as the solution to
    \begin{equation}\label{eq:gateauxpde}
        - \Delta \delta_h + \mathbb{1}_{\{y >0\}} \delta_h = h.
    \end{equation}
\end{corollary}

\begin{proof}
    In view of \eqref{eq:ddpde}, it is clear that if $\lambda^d(\{y = 0 \}) = 0$, then $S$ 
    is G\^ateaux-differentiable 
    in $u$
    and the G\^ateaux derivative is the solution operator for \eqref{eq:gateauxpde}. 
    Further, \ref{enum:Gateaux:ii} trivially implies \ref{enum:Gateaux:iii}.
    It remains to prove 
    that \ref{enum:Gateaux:iii} implies \ref{enum:Gateaux:i}. 
    To this end, 
    we note that if $S'(u;h) = - S'(u;-h)$  for all $h\in L^2(\Omega)$, then \eqref{eq:ddpde} implies
    \begin{equation}\label{eq:randomeq51423514e}
        \mathbb{1}_{\{y = 0\}}\max(0, -S'(u;h))  +  \mathbb{1}_{\{y = 0\}}\max(0, S'(u;h)) 
        =  \mathbb{1}_{\{y = 0\}}|S'(u;h)| = 0
    \end{equation}
    for all $h \in L^2(\Omega)$.
    Consider now a function $\psi \in C^\infty(\R^d)$ with $\psi > 0$ 
    in $\Omega$ and $\psi \equiv 0$ in $\R^d \setminus \Omega$, whose existence is ensured by \cref{lem:randomlemma}. 
    Since $h\in L^2(\Omega)$ was arbitrary, we are allowed to choose
    \begin{equation*}
        h := -\Delta \psi +   \mathbb{1}_{\{y > 0\}} \psi+ \mathbb{1}_{\{y = 0\}} \max(0, \psi) \in L^2(\Omega),
    \end{equation*}
    such that $S'(u;h) = \psi$ by virtue of \eqref{eq:ddpde}. Consequently, we obtain from  
    \eqref{eq:randomeq51423514e} that $\mathbb{1}_{\{y = 0\}}\psi = 0$.
    Since $\psi > 0$ in $\Omega$, this yields $\lambda^d(\{y = 0\})= 0$ as claimed. 
\end{proof}

\section{Bouligand subdifferentials of the control-to-state mapping}\label{sec:bouli}

This section is devoted to the main result of our work, namely the precise characterization of 
the Bouligand subdifferentials of the PDE solution operator $S$ from \cref{prop:basic}.

\subsection{Definitions and basic properties}

We start with the rigorous definition of the Bouligand subdifferential. In the spirit of \cite[Def.~2.12]{okz98}, 
it is defined as the set of limits of Jacobians of differentiable points. However, in infinite dimensions, 
we have of course to distinguish between different topologies underlying this limit process, as already mentioned in the 
introduction. This gives rise to the following

\begin{definition}[Bouligand subdifferentials of $S$]\label{def:bouli}
    Let $u \in L^2(\Omega)$ be given. Denote the set of smooth points of $S$ by 
    \begin{equation*}
        D := \{ v\in L^2(\Omega) : \text{$S: L^2(\Omega) \to Y$ is G\^ateaux-differentiable in $v$}\}.
    \end{equation*} 
    In what follows, we will frequently call points in $D$ \emph{G\^ateaux points}.
    \begin{itemize}
        \item[(i)] The \emph{weak-weak Bouligand subdifferential} of $S$ in $u$ is defined by
            \begin{align*} 
                \partial_{B}^{ww} S(u) 
                := \{ & G \in \mathcal{L}(L^2(\Omega), Y) : 
                    \text{there exists }  (u_n) \subset D \text{ such that}\\
                    & u_n \weakly u \text{ in } L^2(\Omega)
                \text{ and }  S'(u_n)h \weakly G\,h \text{ in } Y  \text{ for all } h \in L^2(\Omega)\}.
            \end{align*}
        \item[(ii)] The \emph{weak-strong Bouligand subdifferential} of $S$ in $u$ is defined by
            \begin{align*} 
                \partial_{B}^{ws} S(u) 
                := \{ & G \in \mathcal{L}(L^2(\Omega), Y) : 
                    \text{there exists }  (u_n) \subset D \text{ such that}\\
                    & u_n \weakly u \text{ in } L^2(\Omega)
                \text{ and }  S'(u_n)h \to G\,h \text{ in } Y  \text{ for all } h \in L^2(\Omega)\}.
            \end{align*}
        \item[(iii)] The \emph{strong-weak Bouligand subdifferential} of $S$ in $u$ is defined by
            \begin{align*} 
                \partial_{B}^{sw} S(u) 
                := \{ & G \in \mathcal{L}(L^2(\Omega), Y) : 
                    \text{there exists }  (u_n) \subset D \text{ such that}\\
                    & u_n \to u \text{ in } L^2(\Omega)
                \text{ and }  S'(u_n)h \weakly G\,h \text{ in } Y  \text{ for all } h \in L^2(\Omega)\}.
            \end{align*}
        \item[(iv)] The \emph{strong-strong Bouligand subdifferential} of $S$ in $u$ is defined by
            \begin{align*} 
                \partial_{B}^{ss} S(u) 
                := \{ & G \in \mathcal{L}(L^2(\Omega), Y) : 
                    \text{there exists }  (u_n) \subset D \text{ such that}\\
                    & u_n \to u \text{ in } L^2(\Omega)
                \text{ and }  S'(u_n)h \to G\,h \text{ in } Y  \text{ for all } h \in L^2(\Omega)\}.
            \end{align*}
    \end{itemize}
\end{definition}

\begin{remark}
    Based on the generalization of Rademacher's theorem to Hilbert spaces (see \cite[Thm.~1.2]{m76}) and 
    the generalization of Alaoglu's theorem to the weak operator topology, one can show
    that $\partial_B^{ww}S(u)$ and $\partial_B^{sw} S(u)$ are non-empty for every $u\in L^2(\Omega)$; 
    see also \cite{cg96}. In contrast to this, it is not clear a priori 
    if $\partial_B^{ws}S(u)$ and $\partial_B^{ss} S(u)$ are non-empty, too. 
    However, \cref{th:endgame} at the end of this section will imply this as a byproduct.
\end{remark}
From the definitions, we obtain the following useful properties.
\begin{lemma}\label{lem:boulibasic}\ 
    \begin{enumerate}[label=(\roman*)]
        \item\label{it:boulibasic1} For all $u \in L^2(\Omega)$ it holds
            \begin{equation*}
                \partial_{B}^{ss} S(u) \subseteq \partial_{B}^{sw} S(u)\subseteq \partial_{B}^{ww} S(u)
                \quad \text{and} \quad
                \partial_{B}^{ss} S(u) \subseteq \partial_{B}^{ws} S(u) \subseteq \partial_{B}^{ww} S(u).
            \end{equation*}
        \item\label{it:boulibasic2} If $S$ is G\^ateaux-differentiable in $u \in L^2(\Omega)$, then it holds
            $S'(u) \in \partial_{B}^{ss} S(u)$.
        \item\label{it:boulibasic3} For all $u \in L^2(\Omega)$ and all $G \in \partial_{B}^{ww} S(u)$, it holds
            \begin{equation*}
                \|G\|_{\LL(L^2(\Omega), Y)} \leq L,
            \end{equation*}
            where $L>0$ is the Lipschitz constant of $S : L^2(\Omega) \to Y$.
    \end{enumerate}
\end{lemma}

\begin{proof}
    Parts \ref{it:boulibasic1} and \ref{it:boulibasic2} immediately follow from the 
    definition of the Bouligand subdifferentials (to see \ref{it:boulibasic2}, 
    just choose $u_n := u$ for all $n$). In order to prove part \ref{it:boulibasic3}, observe that 
    the definition of $\partial_{B}^{ww} S(u)$ implies the existence of 
    a sequence of G\^ateaux points $u_n \in L^2(\Omega)$ such that $u_n \weakly u$ in $L^2(\Omega)$ and 
    $S'(u_n)h \weakly  G h$ in $Y$ for all $h \in L^2(\Omega)$.  For each $n\in\N$, the 
    global Lipschitz continuity of $S$ according to \cref{prop:basic} 
    immediately gives $\|S'(u_n)\|_{\LL(L^2(\Omega), Y)} \leq L$.
    Consequently, the weak lower semi-continuity of the norm implies
    \begin{equation*}
        \|Gh\|_{Y} \leq  \liminf_{n \to \infty} \|S'(u_n)h\|_{Y} \leq L \|h\|_{L^2} 
        \quad \forall \,h \in L^2(\Omega).
    \end{equation*}
    This yields the claim.
\end{proof}
\begin{remark}
    The Bouligand subdifferentials $\partial_B^{ww} S(u)$ and $\partial_B^{sw} S(u)$ do not change 
    if the condition ``$S'(u_n)h \weakly G\,h$ in $Y$ for all $h \in L^2(\Omega)$'' in Definition~\ref{def:bouli}(i) and (iii)
    is replaced with either ``$S'(u_n)h \to G\,h$ in $Z$ for all $h \in L^2(\Omega)$''  or ``$S'(u_n)h \weakly G\,h$ in $Z$ for all $h \in L^2(\Omega)$'', where $Z$ is a 
    normed linear space  
    such that $Y$ is compactly embedded into $Z$, 
    e.g., $Z = H^1(\Omega)$ 
    or $Z = L^2(\Omega)$. This can be seen as follows:
    Suppose that a sequence  $(u_n) \subset D$ is given such that $S'(u_n)h \weakly G\,h$ holds in $Z$ for all $h \in L^2(\Omega)$. 
    Then, by  Lemma~\ref{lem:boulibasic}\ref{it:boulibasic3}, 
    we can find for every $h \in L^2(\Omega)$ a subsequence $(u_{n_k})$ such that $(S'(u_{n_k})h)$ converges weakly in $Y$. From the weak convergence in $Z$, we obtain that the weak limit has to be equal to $G\,h$ independently of the chosen subsequence. Consequently, $S'(u_n)h \weakly G\,h$ in $Y$ for all $h \in L^2(\Omega)$, and we arrive at our original condition. If, conversely, we know that $S'(u_n)h \weakly G\,h$ holds in $Y$ for all $h \in L^2(\Omega)$, then, by the compactness of the embedding $Y \embed Z$, it trivially holds
    $S'(u_n)h \to G\,h$ in $Z$ for all $h \in L^2(\Omega)$. This yields the claim. The case of $S'(u_n)h \to G\,h$ in $Z$ proceeds analogously.
\end{remark}

Next, we show closedness properties of the two strong subdifferentials.
\begin{proposition}[strong-strong-closedness of $\partial_{B}^{ss} S$]\label{prop:stabss}
    Let $u \in L^2(\Omega)$ be arbitrary but fixed. Suppose that
    \begin{enumerate}[label=(\roman*)]
        \item $u_n \in L^2(\Omega)$ and  $G_n \in \partial_{B}^{ss} S(u_n)$ for all $n\in \N$,
        \item $u_n \to u$ in $L^2(\Omega)$,
        \item $G \in \mathcal{L}(L^2(\Omega), Y)$,
        \item $G_n h  \to G h$ in $Y$ for all $h \in L^2(\Omega)$.
    \end{enumerate}
    Then $G$ is an element of $\partial_{B}^{ss} S(u)$.
\end{proposition}
\begin{proof}
    The definition of $\partial_{B}^{ss} S(u_n)$ implies that 
    for all $n \in \mathbb{N}$, one can find a sequence  $(u_{m,n}) \subset L^2(\Omega)$ of G\^ateaux points 
    with associated derivatives $G_{m, n}:= S'(u_{m,n})$ 
    such that $u_{m, n} \to u_n$ in $L^2(\Omega)$ as $m \to \infty$ and
    \begin{equation*}
        G_{m,n}h \to G_n h \text{ in } Y
        \quad \text{for all }h \in L^2(\Omega) \text{ as }m\to \infty. 
    \end{equation*}
    Since $L^2(\Omega)$ is separable, there exists a countable set 
    $\{w_k\}_{k=1}^\infty \subseteq L^2(\Omega)$ that is dense in $L^2(\Omega)$. 
    Because of the convergences derived above, it moreover follows that 
    for all $n \in \mathbb{N}$, there exists an $m_n \in \mathbb{N}$ with
    \begin{equation}\label{eq:separable}
        \| G_{m_n,n}w_k - G_n w_k\|_Y \leq \frac{1}{n}\quad \forall \, k=1,\dots,n 
        \quad \text{ and } \quad \| u_n - u_{m_n, n}\| _{L^2(\Omega)} \leq \frac{1}{n}. 
    \end{equation}
    Consider now a fixed but arbitrary $h \in L^2(\Omega)$, and define
    \begin{equation*}
        h^*_n := \argmin\{\| w_k - h\|_{L^2(\Omega)} : 1 \leq k \leq n\}.
    \end{equation*}
    Then the density property of $\{w_k\}_{k=1}^\infty$ implies $h^*_n \to h$ in $L^2(\Omega)$ 
    as $n \to \infty$, and we may estimate
    \begin{equation*}
        \begin{aligned}
            \|G_{m_n, n}h  -  G h\|_Y
            &\leq \| G_{m_n, n} h_n^*  -  G_n h_n^*\|_Y + 
            \|(G_{m_n, n} - G_n) (h_n^*  -  h)\|_Y + \|G_n h - G h\|_Y\\
            &\leq \frac{1}{n} + \|G_{m_n,n} - G_n\|_{\LL(L^2,Y)}
            \|h_n^* - h\|_Y + \|G_n h - G h\|_Y \to 0 \quad \text{as } n \to \infty,
        \end{aligned}
    \end{equation*}
    where the boundedness of $\|G_{m_n,n} - G_n\|_{\LL(L^2,Y)}$ follows from 
    \cref{lem:boulibasic}\ref{it:boulibasic3}. 
    The above proves that for all $h \in L^2(\Omega)$, we have  
    $G_{m_n,n}h \to  G h$ in $Y$. Since $h\in L^2(\Omega)$
    was arbitrary and the G\^ateaux points $u_{m_n,n}$ 
    satisfy $u_{m_n,n} \to u$ in $L^2(\Omega)$ as $n\to \infty$ by \eqref{eq:separable} and 
    our assumptions, the claim follows from the definition of $\partial_{B}^{ss} S(u)$.
\end{proof}

\begin{proposition}[strong-weak-closedness of $\partial_{B}^{sw} S$]\label{prop:stabsw}
    Let $u \in L^2(\Omega)$ be arbitrary but fixed. Assume that:
    \begin{enumerate}[label=(\roman*)]
        \item $u_n \in L^2(\Omega)$ and  $G_n \in \partial_{B}^{sw} S(u_n)$ for all $n \in \N$,
        \item $u_n \to u$ in $L^2(\Omega)$,
        \item $G \in \mathcal{L}(L^2(\Omega), Y)$,
        \item $G_n h  \weakly G h$ in $Y$ for all $h \in L^2(\Omega)$.
    \end{enumerate}
    Then $G$ is an element of $\partial_{B}^{sw} S(u)$.
\end{proposition}
\begin{proof}
    As in the proof before, for all $n \in \mathbb{N}$ the definition of $\partial_{B}^{sw} S(u_n)$  
    implies the existence of a sequence of G\^ateaux points 
    $u_{m,n} \in L^2(\Omega)$ with associated derivatives $G_{m, n}:= S'(u_{m,n})$ 
    such that $u_{m, n} \to u_n$ in $L^2(\Omega)$ as $m \to \infty$ and
    \begin{equation*}
        G_{m,n}h \weakly G_n h \text{ in } Y
        \quad\text{for all }h \in L^2(\Omega) \text{ as }m\to \infty. 
    \end{equation*}
    Now the compact embedding of $Y$ in $H^1_0(\Omega)$ gives that 
    $G_{m,n}h \to G_n h$ in $H^1_0(\Omega)$ as $m\to \infty$, 
    and we can argue exactly as in the 
    proof of \cref{prop:stabss} to show that there is a diagonal sequence of 
    G\^ateaux points $u_{m_n,n}$ such that $u_{m_n,n} \to u$ in $L^2(\Omega)$ and 
    \begin{equation}\label{eq:H1strong}
        G_{m_n,n} h \to G h \text{ in } H^1_0(\Omega) \quad \text{for every } h\in L^2(\Omega).  
    \end{equation}
    On the other hand, by \cref{lem:boulibasic}\ref{it:boulibasic3}, 
    the operators $G_{m_n,n}$ are uniformly bounded in $\LL(L^2(\Omega);Y)$. 
    Therefore, for an arbitrary but fixed 
    $h\in L^2(\Omega)$, the sequence $\|G_{m_n,n} h\|_Y$ is bounded in $Y$, so that a 
    subsequence converges weakly to some $\eta \in Y$. Because of \eqref{eq:H1strong}, 
    $\eta = G h$ and the uniqueness of the weak limit implies the weak convergence of the whole 
    sequence in $Y$. As $h$ was arbitrary, this implies the assertion.
\end{proof}

\subsection{Precise characterization of the Bouligand subdifferentials}

This section is devoted to an explicit characterization of the different 
subdifferentials in \cref{def:bouli} without the representation 
as (weak) limits of Jacobians of sequences 
of G\^ateaux points. We start with the following lemma, which will be useful 
in the sequel.

\begin{lemma}\label{le:PDElemma}
    Assume that
    \begin{enumerate}[label=(\roman*)]
        \item $j : \R \to \R$ is monotonically increasing and globally Lipschitz continuous,
        \item $(u_n) \subset L^2(\Omega)$ is a sequence with $u_n \rightharpoonup u \in L^2(\Omega)$,
        \item $(\chi_n) \subset L^\infty(\Omega)$ is a sequence satisfying  
            $\chi_n \geq 0$ a.e. in $\Omega$ for all  $n\in\N$ and 
            $\chi_n \weakly^{*} \chi $ in $L^\infty(\Omega)$ for some $\chi\in L^\infty(\Omega)$,
        \item $w_n \in Y$ is the unique solution to 
            \begin{equation}\label{eq:pdewn}
                -\Delta w_n + \chi_n w_n + j(w_n) = u_n,
            \end{equation}
        \item $w \in Y$ is the unique solution to 
            \begin{equation}\label{eq:pdew}
                -\Delta w + \chi  w  + j(w)= u.
            \end{equation}
    \end{enumerate}
    Then it holds that $w_n \rightharpoonup w$ in $Y$, and if we additionally assume that 
    $\chi_n \to  \chi$ pointwise a.e. and $u_n \to u$ strongly in $L^2(\Omega)$, 
    then we even have $w_n \to w$ strongly in $Y$.
\end{lemma}

\begin{proof}
    First note that due to the monotonicity and the global Lipschitz continuity of $j$, the equations
    \eqref{eq:pdewn} and \eqref{eq:pdew}, respectively, admit unique solutions in $Y$
    by the same arguments as in the proof of \cref{prop:basic}.
    Moreover, due to the weak and weak-$\ast$ convergence, the sequences $(u_n)$ and $(\chi_n)$ 
    are bounded in $L^2(\Omega)$ and $L^\infty(\Omega)$, respectively, so that $(w_n)$ is 
    bounded in $Y$. Hence there exists a weakly converging subsequence -- w.l.o.g.\ denoted by the same symbol --
    such that $w_n \weakly \eta$ in $Y$
    and, by the compact embedding $Y\hookrightarrow H_0^1(\Omega)$, $w_n \to \eta$ strongly in $H_0^1(\Omega)$.
    Together with the weak convergence of $u_n$,  
    this allows passing to the limit in \eqref{eq:pdewn} to deduce that $\eta$ satisfies
    \begin{equation*}
        -\Delta \eta + \chi  \eta  + j(\eta)= u.
    \end{equation*}
    As the solution to this equation is unique, we obtain $\eta = w$.
    The uniqueness of the weak limit now gives convergence of the whole sequence, 
    i.e., $w_n \weakly w$ in $Y$.

    To prove the strong convergence under the additional assumptions, note that the difference 
    $w_n-w$ satisfies 
    \begin{equation}\label{eq:pdediff}
        -\Delta (w_n - w) = (u_n - u) + (\chi\,w - \chi_n\, w_n) + (j(w) - j(w_n)). 
    \end{equation}  
    For the first term on the right-hand side of \eqref{eq:pdediff}, we have $u_n \to u$ in $L^2(\Omega)$ by assumption.
    The second term in \eqref{eq:pdediff} is estimated by
    \begin{equation}\label{eq:wchi}
        \|\chi\,w - \chi_n\, w_n\|_{L^2(\Omega)}
        \leq \|\chi_n\|_{L^\infty(\Omega)}\|w - w_n\|_{L^2(\Omega)}
        + \|(\chi - \chi_n)w\|_{L^2(\Omega)}.
    \end{equation}
    The first term in \eqref{eq:wchi} converges to zero due to $w_n \weakly w$ in $Y$ and the compact embedding, while the convergence of the second term follows 
    from the pointwise convergence of $\chi_n$ in combination with Lebesgue's dominated convergence theorem. 
    The global Lipschitz continuity of $j$ and the strong convergence of $w_n \to w$ in $L^2(\Omega)$ 
    finally also give $j(w_n) \to j(w)$ in $L^2(\Omega)$. 
    Therefore, the right-hand side in \eqref{eq:pdediff} converges to zero in $L^2(\Omega)$. As 
    $-\Delta$ induces the norm on $Y$, we thus obtain 
    the desired strong convergence.
\end{proof}

By setting $j(x) = \max(0, x)$ and $\chi_n \equiv \chi \equiv 0$, we obtain 
as a direct consequence of the preceding lemma the following weak continuity of $S$.
\begin{corollary}\label{cor:Vollstetig}
    The solution operator $S : L^2(\Omega) \to Y$ is weakly continuous, i.e., 
    \begin{equation*}
        u_n \weakly u \text{ in }L^2(\Omega) \quad \Longrightarrow \quad  
        S(u_n) \weakly S(u) \text{ in } Y.
    \end{equation*}
\end{corollary}

We will see in the following that all elements of the subdifferentials in 
\cref{def:bouli} have a similar structure. To be precise, they are solution operators of linear 
elliptic PDEs of a particular form. 
\begin{definition}[linear solution operator $G_\chi$]\label{def:Gchi}
    Given a function $\chi \in L^\infty(\Omega)$ with $\chi \geq 0$, we define 
    the operator $G_\chi \in \LL(L^2(\Omega), Y)$ to be the solution operator of the 
    linear equation 
    \begin{equation}\label{eq:linpdechi}
        -\Delta \eta +\chi\, \eta = h.
    \end{equation}  
\end{definition}

We first address necessary conditions for an operator in $\LL(L^2(\Omega),Y)$ to be an element of 
the Bouligand subdifferentials. Afterwards we will show that these conditions are also sufficient, 
which is more involved compared to their necessity.
\begin{proposition}[necessary condition for $\partial_{B}^{ww} S(u)$]\label{prop:necessaryweak}
    Let $u \in L^2(\Omega)$ be arbitrary but fixed and set $y:=S(u)$. 
    Then for every $G \in  \partial_{B}^{ww} S(u)$ there exists a unique $\chi \in L^\infty(\Omega)$
    satisfying
    \begin{equation}\label{eq:chiweakstar}
        0 \leq \chi \leq 1 \text{ a.e. in }\Omega, \quad \chi = 1 \text{ a.e. in } \{y > 0\},
        \quad \text{and} \quad \chi = 0 \text{ a.e. in } \{y < 0\}
    \end{equation}
    such that $G = G_\chi$.
\end{proposition}

\begin{proof}
    If $G \in  \partial_{B}^{ww} S(u)$ is arbitrary but fixed, 
    then there exists a sequence of G\^ateaux points $u_n \in L^2(\Omega) $ such that  
    $u_n \rightharpoonup u$ in $L^2(\Omega)$ and $S'(u_n)h \weakly G h$ in $Y$ 
    for all $h \in L^2(\Omega)$. 
    Now, let $h\in L^2(\Omega)$ be arbitrary but fixed and abbreviate  
    $y_n := S(u_n)$, $\delta_{h,n} := S'(u_n)h$, and $\chi_n :=  \mathbb{1}_{\{y_n >0\}}$.
    Then we know from \cref{cor:Vollstetig} that $y_n \rightharpoonup y$ in $Y$ 
    and from \cref{cor:GateauxCrit} that $\delta_{h,n} = G_{\chi_n} h$.
    Moreover, from the Banach--Alaoglu Theorem it follows that -- 
    after transition to a subsequence (which may be done independently of $h$) --
    it holds that $\chi_n \weakly^{*} \chi$ in $L^\infty(\Omega)$.
    From the weak-$\ast$ closedness of the set
    $\{\xi\in L^\infty(\Omega): 0\leq \xi\leq 1 \text{ a.e.\ in }\Omega\}$
    we obtain that $0 \leq \chi \leq 1$ holds a.e.\ in $\Omega$. 
    Further,  the definition of $\chi_n$ and the convergences $\chi_n  \weakly^{*} \chi$ in $L^\infty(\Omega)$ and $y_n \to y$ in $L^2(\Omega)$ yield
    \begin{equation*}
        0 = \int_\Omega \chi_n \min(0, y_n) - (1 - \chi_n)\max(0, y_n) \d x 
        \to \int_\Omega \chi \min(0, y) - (1 - \chi)\max(0, y) \d x = 0.
    \end{equation*}
    Due to the sign of the integrand, the above yields  $\chi \min(0, y) - (1 - \chi)\max(0, y) = 0$ a.e.\ in $\Omega$, and this entails $\chi = 1$ a.e.\ in $\{y>0\}$ and $\chi = 0$  a.e.\ in $\{y<0\}$. 
    This shows that $\chi$ satisfies \eqref{eq:chiweakstar}.
    From Lemma \ref{le:PDElemma}, we may deduce that $\delta_{h,n} \weakly G_\chi h$ in $Y$. 
    We already know, however, that $\delta_{h,n} = S'(u_n)h \rightharpoonup G h$ in $Y$. 
    Consequently, since $h$ was arbitrary, $G = G_\chi$, and the existence claim is proven. 
    It remains to show that $\chi$ is unique. To this end, assume that there are two different 
    functions $\chi, \tilde{\chi} \in L^\infty(\Omega)$ with $G = G_\chi = G_{\tilde\chi}$. 
    If we then consider a function $\psi \in C^\infty(\R^d)$ with $\psi > 0$ in $\Omega$ 
    and $\psi \equiv 0$ in $\R^d \setminus \Omega$ (whose existence is ensured by 
    \cref{lem:randomlemma}) and define $h_{\psi} :=  -\Delta \psi +\chi \psi \in L^2(\Omega)$, 
    then we obtain $\psi = G_\chi h_\psi = G_{\tilde\chi} h_\psi$, which gives rise to 
    \begin{equation*}
        -\Delta \psi+\chi \psi= h_{\psi} =  -\Delta \psi+\tilde{\chi} \psi.
    \end{equation*}
    Subtraction now yields $(\chi - \tilde{\chi})\psi = 0$ a.e. in $\Omega$ and, 
    since $\psi > 0$ in $\Omega$, this yields $\chi \equiv \tilde{\chi}$.
\end{proof}

\begin{proposition}[necessary condition for $\partial_{B}^{ws} S(u)$]\label{prop:necessarystrong}
    Let $u \in L^2(\Omega)$ be arbitrary but fixed with $y = S(u)$. 
    Then for every $G \in  \partial_{B}^{ws} S(u)$  
    there exists a unique function $\chi \in L^\infty(\Omega)$ satisfying
    \begin{equation}\label{eq:chistrong}
        \chi \in \{0, 1\} \text{ a.e. in }\Omega, \quad \chi = 1 \text{ a.e. in } \{y > 0\},
        \quad \text{and} \quad \chi = 0 \text{ a.e. in } \{y < 0\}
    \end{equation}
    such that $G = G_\chi$.
\end{proposition}

\begin{proof} 
    Let $G \in  \partial_{B}^{ws} S(u)$ be fixed but arbitrary.
    Since $\partial_{B}^{ws} S(u) \subseteq \partial_{B}^{ww} S(u)$, the preceding proposition 
    yields that there is a unique function $\chi$ satisfying \eqref{eq:chiweakstar} such that 
    $G = G_\chi$. It remains to prove that $\chi$ only takes values in $\{0,1\}$. 
    To this end, first observe that the definition of $\partial_{B}^{ws} S(u)$ implies the 
    existence of a sequence of G\^ateaux points $(u_n)\subset L^2(\Omega)$ such that 
    $u_n \weakly u$ in $L^2(\Omega)$ and $S'(u_n) h \to G h$ in $Y$
    for every $h \in L^2(\Omega)$, where, according to 
    \cref{cor:GateauxCrit}, $S'(u_n) = G_{\chi_n}$ with $\chi_n := \mathbb{1}_{\{y_n >0\}}$.
    As in the proof of \cref{prop:necessaryweak}, we choose 
    the special direction $h_\psi := -\Delta \psi + \chi  \psi \in L^2(\Omega)$,
    where $\psi \in C^\infty(\R^d)$ is again a function with $\psi > 0$ in $\Omega$ 
    and $\psi \equiv 0$ in $\R^d \setminus \Omega$.
    Then $G h_\psi = \psi$, and the strong convergence of $G_{\chi_n} h_\psi$ to $G h_\psi$
    in $Y$ allows passing to a subsequence to obtain 
    $\Delta G_{\chi_n}h_\psi \to \Delta \psi$ and $G_{\chi_n} h_\psi \to \psi$ pointwise a.e. in $\Omega$. 
    From the latter, it follows that for almost all $x \in \Omega$ there exists an $N \in \mathbb{N}$
    (depending on $x$) with $G_{\chi_n}h_\psi(x) > 0$ for all $n \geq N$, and consequently
    \begin{equation*}
        \lim_{n \to \infty} \chi_n(x) 
        = \lim_{n \to \infty} \frac{h_\psi(x) + \Delta G_{\chi_n}h_\psi(x)}{G_{\chi_n}h_\psi(x)} 
        = \frac{h_\psi(x) + \Delta \psi(x)}{\psi(x)}  = \chi(x) \text{ for a.a. } x \in \Omega. 
    \end{equation*}
    But, as $\chi_n$ takes only the values $0$ and $1$ for all $n\in\N$, 
    pointwise convergence almost everywhere is only possible 
    if $\chi \in \{0, 1\}$ a.e. in $\Omega$. This proves the claim. 
\end{proof}

As an immediate consequence of the last two results, we obtain:

\begin{corollary}\label{cor:gateaux}
    If $S$ is G\^ateaux-differentiable in a point $u \in L^2(\Omega)$, then it holds
    \begin{equation*}
        \partial_{B}^{ss} S(u) =\partial_{B}^{sw} S(u) = \partial_{B}^{ws} S(u)
        = \partial_{B}^{ww} S(u) = \{S'(u)\}. 
    \end{equation*}
\end{corollary}

\begin{proof}
    The inclusion $\supseteq$ was already proved in \cref{lem:boulibasic}.
    The reverse inclusion follows immediately from \cref{prop:necessaryweak,prop:necessarystrong}, and the fact that in a G\^ateaux point
    there necessarily holds $\lambda^d(\{y = 0\}) = 0$ (see \cref{cor:GateauxCrit}).
\end{proof}

\begin{remark}
    Note that even in finite dimensions, the
    Bouligand and the Clarke subdifferential can contain operators 
    other than the G\^ateaux derivative in a G\^ateaux point; see, e.g., \cite[Ex.~2.2.3]{Clarke:1990}. 
    Thus, \cref{cor:gateaux} shows that, in spite of its non-differentiability, the solution operator $S$
    is comparatively well-behaved.
\end{remark}

\begin{remark}\label{rem:subdiffmax}
    Similarly to \cref{thm:rabl}, 
    where the directional derivative of the $\max$-function appears, 
    \cref{prop:necessaryweak,prop:necessarystrong} show that elements of  
    $\partial_{B}^{ww} S(u)$ and $\partial_{B}^{ws} S(u)$ are characterized by PDEs which involve 
    a pointwise measurable selection $\chi$ of the set-valued a.e.-defined functions $\partial_B \max(0, \cdot)(y) : \Omega  \rightrightarrows [0,1]$  and $\partial_c \max(0, \cdot)(y) : \Omega  \rightrightarrows [0,1]$, respectively. Here, $\partial_B \max(0, \cdot)$ and $\partial_c \max(0, \cdot)$ denote the Bouligand and the convex subdifferential of the function $\R \ni x \mapsto \max(0, x)\in \R$, respectively. 
\end{remark}

Now that we have found necessary conditions that elements of the subdifferentials 
$\partial_{B}^{ws} S(u)$ and $\partial_{B}^{ww} S(u)$ have to fulfill, we turn to 
sufficient conditions which guarantee that a certain linear operator is an 
element of these subdifferentials. Here we focus on the subdifferentials 
$\partial_{B}^{ss} S(u)$ and $\partial_{B}^{sw} S(u)$. It will turn out that 
a linear operator is an element of these subdifferentials if it is of the form 
$G_\chi$ with $\chi$ as in \eqref{eq:chiweakstar} and \eqref{eq:chistrong}, respectively.
Thanks to \cref{lem:boulibasic}\ref{it:boulibasic1} and the necessary conditions in
\cref{prop:necessaryweak,prop:necessarystrong}, this will  
finally give a sharp characterization of all Bouligand subdifferentials 
in \cref{def:bouli}; 
see \cref{th:endgame} below. We start with the following preliminary result.

\begin{lemma}\label{lem:preliminarycharacterization}
    Let $u \in L^2(\Omega)$ be arbitrary but fixed and write $y := S(u)$.
    Assume that $\varphi \in Y$ is a function satisfying
    \begin{equation}\label{eq:nullsep}
        \lambda^d(\{y =0\} \cap \{\varphi = 0\}) = 0
    \end{equation}
    and define $\chi \in L^\infty(\Omega)$ via
    \begin{equation}\label{eq:defchi}
        \chi := \mathbb{1}_{\{y>0\}} + \mathbb{1}_{\{y=0\} } \mathbb{1}_{  \{\varphi > 0\}}.
    \end{equation}  
    Then $G_\chi$ as in Definition \ref{def:Gchi} is an element of the strong-strong Bouligand subdifferential $\partial_B^{ss} S(u)$.
\end{lemma}

\begin{proof}
    We have to construct sequences of G\^ateaux points converging strongly to $u$ such that 
    also the corresponding G\^ateaux derivatives in an arbitrary direction $h\in L^2(\Omega)$ converge strongly 
    in $Y$ to $G_\chi h$. For this purpose, set  $y_\varepsilon := y + \varepsilon \varphi$,
    $\varepsilon \in (0, 1)$, and 
    $u_\varepsilon := - \Delta y_\varepsilon + \max(0, y_\varepsilon) \in L^2(\Omega)$. 
    Then we obtain $S( u_\varepsilon) = y_\varepsilon$, $y_\varepsilon \to y$ in $Y$ and 
    $u_\varepsilon \to u$ in $L^2(\Omega)$ as $\varepsilon \to 0$. 
    Choose now arbitrary but fixed representatives of $y$ and $\varphi$ and 
    define $Z:= \{y = 0\} \cap \{\varphi = 0\}$. Then for all $\varepsilon \neq \varepsilon'$ 
    and  all $x \in \Omega$, we have
    \begin{equation*}
        y(x) + \varepsilon \varphi(x) = 0 \quad\text{and}\quad  y(x) + \varepsilon' \varphi(x)=0 \iff x \in  Z,
    \end{equation*}
    i.e., the sets in the collection 
    $(\{y + \varepsilon \varphi =0\} \setminus Z)_{\varepsilon \in (0, 1)}$ are disjoint 
    (and obviously Lebesgue measurable). 
    Furthermore, the underlying measure space $(\Omega, \mathcal{L}(\Omega), \lambda^d)$ is finite. 
    Thus, we may apply \cref{MassEffect} to obtain a $\lambda^1$-zero set 
    $N \subset (0, 1)$ such that
    \begin{equation*}
        \lambda^d(\{y + \varepsilon \varphi = 0\} )  
        \leq \lambda^d (Z) + \lambda^d(\{y + \varepsilon \varphi =0\} \setminus Z) = 0
    \end{equation*}
    for all $\varepsilon \in E := (0, 1) \setminus N$. 
    According to \cref{cor:GateauxCrit}, this implies that $S$ is G\^ateaux-differentiable 
    in $u_\varepsilon$ for all $\varepsilon \in E$ with 
    $S'(u_\varepsilon) = G_{\chi_\varepsilon}$ where 
    $\chi_\varepsilon := \mathbb{1}_{\{y + \varepsilon \varphi >0\}}$.
    Consider now an arbitrary but fixed sequence $(\varepsilon_n) \subset E$ 
    with $\varepsilon_n \to 0$ as $n \to \infty$ 
    and fix for the time being a direction $h \in L^2(\Omega)$. 
    Then $\delta_{\varepsilon_n} := S'(u_{\varepsilon_n}) h = G_{\chi_{\varepsilon_n}} h$ satisfies
    \begin{equation*}
        -\Delta \delta_{\varepsilon_n}  +  \chi_{\varepsilon_n} \delta_{\varepsilon_n} = h, 
    \end{equation*}
    and it holds that
    \begin{equation*}
        \lim_{n \to \infty} \chi_{\varepsilon_n}(x) = \chi(x) \text{ f.a.a. } x\in \Omega 
        \quad \text{and} \quad 
        \chi_{\varepsilon_n} \weakly{*} \chi \text{ in }L^\infty(\Omega),
    \end{equation*}
    where $\chi$ is as defined in \eqref{eq:defchi}. Note that according to assumption 
    \eqref{eq:nullsep}, the case $y(x)  = \varphi(x) = 0$ is negligible here. 
    Using \cref{le:PDElemma}, we now obtain that 
    $\delta_{\varepsilon_n} = S'(u_{\varepsilon_n})h$ converges strongly in $Y$ 
    to $G_\chi h$. Since $h\in L^2(\Omega)$ was arbitrary, this proves the claim. 
\end{proof}

In the following, we successively sharpen the assertion of 
\cref{lem:preliminarycharacterization} by means of \cref{le:PDElemma}
and the approximation results for characteristic functions proven in \cref{sec:approx}.

\begin{proposition}[sufficient condition for $\partial_B^{ss} S(u)$]\label{prop:suffstrong}
    Let $u \in L^2(\Omega)$ be arbitrary but fixed and set $y := S(u)$.
    If $\chi \in L^\infty(\Omega)$ satisfies
    \begin{equation}\label{eq:chicondstrong}
        \chi \in \{0, 1\} \text{ a.e. in }\Omega, \quad \chi = 1 \text{ a.e. in } \{y > 0\},
        \quad \text{and} \quad \chi = 0 \text{ a.e. in } \{y < 0\},
    \end{equation}
    then $G_\chi \in \partial_B^{ss} S(u)$.
\end{proposition}

\begin{proof}
    Since $\chi \in L^\infty(\Omega)$ takes only the values $0$ and $1$, 
    there exists a Lebesgue measurable set $B  \subseteq \Omega$ with $\chi = \mathbb{1}_B$.  
    We now proceed in two steps:
    \begin{itemize}
        \item[(i)] If $B\subseteq\Omega$ is open, then we know from  
            \cref{lem:approx}\ref{it:approxopen} that there exists a sequence $(\varphi_n) \subset Y$ 
            of functions satisfying \eqref{eq:nullsep} such that 
            \begin{equation}\label{eq:einsphi}
                \qquad \mathbb{1}_{\{\varphi_n > 0\}} \weakly^{*} \mathbb{1}_B 
                \text{ in }L^\infty(\Omega) 
                \quad \text{and}\quad 
                \mathbb{1}_{\{\varphi_n > 0\}} \to \mathbb{1}_{B}\text{ pointwise a.e. in }\Omega.
            \end{equation}
            Let us abbreviate $\chi_n := \mathbb{1}_{\{y>0\}} 
            + \mathbb{1}_{\{y=0\} }\mathbb{1}_{\{\varphi_n > 0\}} \in L^\infty(\Omega)$.
            Then, thanks to \eqref{eq:chicondstrong} and \eqref{eq:einsphi}, we arrive at
            \begin{equation}\label{eq:chikonv}
                \chi_n \weakly^{*}  \chi \text{ in }L^\infty(\Omega) 
                \quad \text{and}\quad 
                \chi_n \to  \chi \text{ pointwise a.e. in }\Omega.  
            \end{equation}
            Moreover, due to \cref{lem:preliminarycharacterization}, we know that
            $G_{\chi_n} \in  \partial_B^{ss} S(u)$ for all $n\in\N$ and, 
            from \cref{le:PDElemma} and \eqref{eq:chikonv}, we obtain 
            $G_{\chi_n} h  \to G_{\chi} h$ strongly in $Y$ for all $h \in L^2(\Omega)$. 
            \Cref{prop:stabss} now gives $G_{\chi} \in \partial_B^{ss} S(u)$ as claimed. 
        \item[(ii)] Assume now that $\chi = \mathbb{1}_B$ for some Lebesgue measurable set $B \subseteq \Omega$. 
            Then \cref{lem:approx}\ref{it:approxborel} implies the existence of a sequence of 
            open sets $B_n \subseteq \Omega$ such that
            \begin{equation}\label{eq:Beins}
                \mathbb{1}_{B_n} \overset{*}{\rightharpoonup} \mathbb{1}_B 
                \text{ in }L^\infty(\Omega)
                \quad \text{and}\quad \mathbb{1}_{B_n} \to  \mathbb{1}_{B}
                \text{ pointwise a.e. in }\Omega.  
            \end{equation}
            Similarly to above, we abbreviate $\chi_n := \mathbb{1}_{\{y>0\}} 
            + \mathbb{1}_{\{y=0\} }\mathbb{1}_{B_n} \in L^\infty(\Omega)$ so that 
            \eqref{eq:Beins} again yields the convergence in \eqref{eq:chikonv}. 
            Since from (i) we know  that $G_{\chi_n} \in \partial_B^{ss} S(u)$ for all $n\in\N$, 
            we can argue completely analogously to (i) to prove $G_{\chi} \in \partial_B^{ss} S(u)$ 
            in the general case. 
            \qedhere
    \end{itemize}
\end{proof}

\begin{proposition}[sufficient condition for $\partial_B^{sw} S(u)$]\label{prop:suffweak}
    Let $u \in L^2(\Omega)$ be arbitrary but fixed and $y := S(u)$.
    If $\chi \in L^\infty(\Omega)$ satisfies
    \begin{equation}\label{eq:chicondweak}
        0 \leq \chi \leq 1 \text{ a.e. in }\Omega, \quad \chi = 1 \text{ a.e. in } \{y > 0\},
        \quad \text{and} \quad \chi = 0 \text{ a.e. in } \{y < 0\}, 
    \end{equation}
    then $G_\chi \in \partial_B^{sw} S(u)$.
\end{proposition}

\begin{proof}
    We again proceed in two steps:
    \begin{itemize}
        \item[(i)] If $\chi$ is a simple function of the form
            $
            \chi := \sum_{k=1}^K c_k \mathbb{1}_{B_k}
            $
            with $c_k \in (0, 1]$ for all $k$, $K \in \mathbb{N}$, and $B_k \subseteq \Omega$ Lebesgue measurable 
            and mutually disjoint, then we know from \cref{lem:approx}\ref{it:approxsimple} 
            that there exists a sequence of Lebesgue measurable sets $A_n \subseteq \Omega$ such that
            $\mathbb{1}_{A_n} \weakly^{*} \chi$ in $L^\infty(\Omega)$.
            In view of \eqref{eq:chicondweak}, this yields
            \begin{equation*}
                \chi_n := \mathbb{1}_{\{y>0\}} + \mathbb{1}_{\{y=0\} } \mathbb{1}_{A_n} 
                \weakly^{*}  \chi   \text{ in }L^\infty(\Omega)
            \end{equation*}    
            so that, by \cref{le:PDElemma}, we obtain $G_{\chi_n} h \rightharpoonup  G_{\chi} h$
            in $Y$ for all $h \in L^2(\Omega)$. Moreover, 
            from \cref{prop:suffstrong}, we already know that
            $G_{\chi_n} \in \partial_B^{ss} S(u) \subseteq \partial_B^{sw} S(u)$ for all $n\in\N$.
            Therefore \cref{prop:stabsw} gives $G_{\chi} \in \partial_B^{sw} S(u)$ as claimed. 
        \item[(ii)] For an arbitrary but fixed $\chi \in L^\infty(\Omega)$ satisfying \eqref{eq:chicondweak}, 
            measurability implies the existence of a sequence of simple functions 
            satisfying \eqref{eq:chicondweak} and converging pointwise a.e.\ to $\chi$.
            (Note that the pointwise projection of a simple function onto the set of functions satisfying \eqref{eq:chicondweak} remains simple as $y$ is fixed and measurable.)
            Since pointwise a.e.\ convergence and a uniform bound in $L^\infty(\Omega)$ imply \mbox{weak-$*$} convergence in $L^\infty(\Omega)$, we can now apply (i) and again 
            \cref{le:PDElemma} and \cref{prop:stabsw} to obtain the claim. 
            \qedhere
    \end{itemize}
\end{proof}

Thanks to \cref{lem:boulibasic}\ref{it:boulibasic1}, the necessary conditions 
in \cref{prop:necessaryweak,prop:necessarystrong}, respectively, 
in combination with the sufficient conditions 
in \cref{prop:suffweak,prop:suffstrong}, respectively, immediately 
imply the following sharp characterization of the Bouligand 
subdifferentials of $S$.

\begin{theorem}[precise characterization of the subdifferentials of $S$]\label{th:endgame}
    Let $u \in L^2(\Omega)$ be arbitrary but fixed and set $y:= S(u)$. Then:
    \begin{itemize}
        \item[(i)] It holds $\partial_B^{ws} S(u) =  \partial_B^{ss} S(u)$.
            Moreover, $G \in  \partial_{B}^{ss} S(u)$ if and only if there exists a function 
            $\chi \in L^\infty(\Omega)$ satisfying 
            \begin{equation*}
                \qquad \chi \in \{0, 1\} \text{ a.e. in }\Omega, \quad \chi = 1 \text{ a.e. in } \{y > 0\},
                \quad \text{and} \quad \chi = 0 \text{ a.e. in } \{y < 0\}
            \end{equation*}
            such that $G = G_\chi$.
            Furthermore, for each $G \in  \partial_{B}^{ss} S(u)$ the associated $\chi$ is unique. 
        \item[(ii)] It holds $\partial_B^{ww} S(u) =   \partial_B^{sw} S(u)$.
            Moreover, $G \in  \partial_{B}^{sw} S(u)$ if and only if there exists a function 
            $\chi \in L^\infty(\Omega)$ satisfying 
            \begin{equation*}
                \qquad 0 \leq \chi \leq 1 \text{ a.e. in }\Omega, \quad \chi = 1 \text{ a.e. in } \{y > 0\},
                \quad \text{and} \quad \chi = 0 \text{ a.e. in } \{y < 0\}
            \end{equation*}
            such that $G = G_\chi$.
            Furthermore, for each $G \in  \partial_{B}^{sw} S(u)$ the associated $\chi$ is unique. 
    \end{itemize}
\end{theorem}

\begin{remark}\label{rem:subdiffmax2}
    \Cref{th:endgame} shows that it does not matter whether we use the weak or the 
    strong topology for the approximating sequence $u_n \in L^2(\Omega)$ 
    in the definition of the subdifferential; only the choice of 
    the operator topology makes a difference. 
    We further see that the elements of the strong resp.~weak Bouligand subdifferential 
    are precisely those operators $G_\chi$ generated by a function $\chi\in L^\infty(\Omega)$ 
    that is obtained from a pointwise measurable selection of the 
    Bouligand resp.~convex subdifferential of the $\max$-function, 
    cf.~\cref{rem:subdiffmax}.
    Note that for all $\chi_1, \chi_2 \in L^\infty(\Omega)$ with $\chi_1 \geq 0$, $\chi_2 \geq 0$ a.e.\ in $\Omega$, all $\alpha \in [0,1]$, and all $\eta \in Y$, it holds 
    \begin{equation*}
        \alpha G_{\chi_1}^{-1}(\eta) + (1-\alpha)G_{ \chi_2}^{-1}(\eta) = -\Delta \eta + \alpha \chi_1 \eta + (1-\alpha) \chi_2 \eta = G_{\alpha \chi_1 + (1-\alpha)\chi_2}^{-1}(\eta).
    \end{equation*}
    This implies that the set $\{G^{-1} : G \in \partial_B^{sw} S(u)\}$ is convex and contains the convex hull of the set  $\{G^{-1} : G \in \partial_B^{ss} S(u)\}$. We point out that the convex combination of two elements of, e.g.,  $\partial_B^{sw} S(u)$ is typically not an element of $\partial_B^{sw} S(u)$ (due to the bilinear term $\chi \eta$ in the definition of $G_\chi$). The above ``convexification effect'' appears only when we consider the inverse operators.
\end{remark}

\section{First-order optimality conditions}\label{sec:fon}

In this section we turn our attention back to the optimal control problem \eqref{eq:p}, 
where we are mainly interested in the derivation of first-order necessary optimality conditions 
involving dual variables. Due to the non-smoothness of the 
control-to-state mapping $S$ caused by the $\max$-function in \eqref{eq:pde}, 
the standard procedure based on the adjoint of the G\^ateaux derivative of $S$ cannot be applied. 
Instead, regularization and relaxation methods are frequently used to derive optimality conditions
as in, e.g., \cite{barbu}.
We will follow the same approach and derive an optimality system in this way 
in the next subsection. Since the arguments are rather standard, we keep the discussion concise;
the main issue here is to carry out the passage to the limit in the topology of $Y$.
We again emphasize that the optimality conditions themselves are not remarkable at all. 
However, in \cref{sec:Bstat}, we will give a new interpretation of the 
optimality system arising through regularization by means of the Bouligand subdifferentials 
from \cref{sec:bouli} (cf.~\cref{theorem:summary}), which is the main 
result of this section.

\subsection{Regularization and passage to the limit}\label{sec:limit}

For the rest of this section, let $\bar{u}\in L^2(\Omega)$ be an arbitrary local 
minimizer for \eqref{eq:p}. We follow a widely used approach (see, e.g., \cite{mp84}) and define 
our regularized optimal control problem as 
\begin{equation}\tag{P$_\varepsilon$}\label{eq:peps}
    \left.
        \begin{aligned}
            &\min\nolimits_{u \in L^2(\Omega), y \in H_0^1(\Omega)}  \quad J(y,u) + \frac{1}{2}\,\|u - \bar{u}\|_{L^2(\Omega)}^2\\
            &\text{s.t.} \quad -\Delta y + \mx_\varepsilon(y) = u \; \text{ in }\Omega
        \end{aligned}
    \qquad \right\}
\end{equation}
with a regularized version of the $\max$-function satisfying the following assumptions.
\begin{assumption}\label{assu:maxeps}
    The family of functions $\mx_\varepsilon: \R \to \R$ satisfies
    \begin{enumerate}[label=$(\roman*)$]
        \item\label{it:maxeps1} $\mx_\varepsilon \in C^1(\R)$ for all $\varepsilon > 0$;
        \item\label{it:maxeps2} there is a constant $C>0$ such that 
            $|\mx_\varepsilon(x) - \max(0,x)| \leq C\,\varepsilon$ for all $x\in \R$;
        \item\label{it:maxeps3} for all $x\in \R$ and all $\varepsilon > 0$, 
            there holds $0 \leq \mx_\varepsilon'(x) \leq 1$;
        \item\label{it:maxeps4} for every $\delta > 0$, the sequence 
            $(\mx_\varepsilon')_{\varepsilon > 0}$ converges uniformly to $1$ on $[\delta, \infty)$ and 
            uniformly to $0$ on $(-\infty, -\delta]$ as $\varepsilon \to 0^+$.
    \end{enumerate}
\end{assumption}

There are numerous possibilities to construct families of functions satisfying 
\cref{assu:maxeps}; we only refer to the regularized $\max$-functions used in \cite{mp84, sw13}.
As for the $\max$-function, we will denote the Nemytskii operator 
associated with $\mx_\varepsilon$ by the same symbol.

\begin{lemma}\label{lem:Sepsfrechet}
    For every $u \in L^2(\Omega)$, there exists a unique solution $y_\varepsilon \in Y$ 
    of the PDE in \eqref{eq:peps}. The associated solution operator 
    $S_\varepsilon : L^2(\Omega) \to Y$ is weakly continuous and Fr\'echet-differentiable. Its derivative 
    at $u\in L^2(\Omega)$ in direction $h\in L^2(\Omega)$ is given by the unique solution $\delta \in Y$ to
    \begin{equation}\label{eq:linpdeeps}
        -\Delta \delta + \mx_\varepsilon'(y_\varepsilon)\delta = h,
    \end{equation}  
    where $y_\varepsilon = S_\varepsilon(u)$.
\end{lemma}

\begin{proof}
    The arguments are standard. The monotonicity of $\mx_\varepsilon$ by 
    \cref{assu:maxeps}\ref{it:maxeps3} yields the existence of a unique solution, 
    and bootstrapping implies that this is an element of $Y$. 
    The weak continuity of $S_\varepsilon$ follows from \cref{le:PDElemma} 
    in exactly the same way as \cref{cor:Vollstetig}.
    Due to \cref{assu:maxeps}\ref{it:maxeps1} and \ref{it:maxeps3}, 
    the Nemytskii operator associated with $\mx_\varepsilon$ is continuously 
    Fr\'echet-differentiable from $H^1_0(\Omega)$ to $L^2(\Omega)$ and, 
    owing to the non-negativity of $\mx_\varepsilon'$, the linearized equation in 
    \eqref{eq:linpdeeps} admits a unique solution $\delta\in Y$ for every $h\in L^2(\Omega)$. 
    The implicit function theorem then yields the claimed differentiability.
\end{proof}

\begin{lemma}\label{lemma:regconv}
    There exists a constant $c > 0$ such that for all $u\in L^2(\Omega)$, there holds 
    \begin{equation}\label{eq:Seps}
        \|S(u) - S_\varepsilon(u)\|_Y \leq c\, \varepsilon \quad \forall\, \varepsilon > 0.
    \end{equation}
    Moreover, for every sequence $u_n \in L^2(\Omega)$ with $u_n \to u$ in $L^2(\Omega)$ 
    and every sequence $\varepsilon_n \to 0^+$, there exists a subsequence 
    $(n_k)_{k\in\N}$ and an operator $G \in \partial_{B}^{sw} S(u)$ such that
    \begin{equation*}
        S'_{\varepsilon_{n_k}}(u_{n_k})h\rightharpoonup G h 
        \text{ in } Y\quad \forall \, h \in L^2(\Omega).
    \end{equation*}
\end{lemma}

\begin{proof}
    Given  $u \in L^2(\Omega)$, let us set $y:= S(u)$ and $y_\varepsilon := S_\varepsilon(u)$. 
    Then it holds that
    \begin{equation}\label{eq:diffpdeeps}
        -\Delta ( y - y_\varepsilon) +  \max(0, y) -  \max(0, y_\varepsilon)  
        = \mx_\varepsilon(y_\varepsilon) - \max(0, y_\varepsilon).
    \end{equation}
    Testing this equation with $y - y_\varepsilon$ and employing the monotonicity of 
    the $\max$-operator and \cref{assu:maxeps}\ref{it:maxeps2} gives
    $\|y-y_\varepsilon\|_{H^1(\Omega)} \leq c\,\varepsilon$. Then, 
    thanks to the Lipschitz continuity of the $\max$-function and again 
    \cref{assu:maxeps}\ref{it:maxeps2}, a bootstrapping 
    argument completely analogous to that in the proof of Proposition \ref{prop:basic}
    yields \eqref{eq:Seps}, cf.~\eqref{eq:diffpdeeps}.

    To obtain the second part of the lemma, 
    let $(u_n) \subset L^2(\Omega)$ and $(\varepsilon_n) \subset (0, \infty)$ be sequences with 
    $u_n \to u$ in $L^2(\Omega)$ and $\varepsilon_n \to 0^+$. 
    Then \eqref{eq:Seps} and \cref{prop:basic} imply
    \begin{equation*}
        \|S_{\varepsilon_n}(u_n) - S(u)\|_{Y} \leq C  \varepsilon_n + \|S(u_n) - S(u)\|_{Y} \to 0  
    \end{equation*}
    as $n \to \infty$, i.e., $y_n := S_{\varepsilon_n}(u_n) \to y:= S(u)$ in $Y$. 
    Now, given an arbitrary but fixed direction $h \in L^2(\Omega)$, 
    we know that the derivative $\delta_n := S'_{\varepsilon_{n}}(u_{n})h$ is characterized by 
    \begin{equation*}
        - \Delta \delta_n + \mx_{\varepsilon_n}'(y_n)\delta_n = h. 
    \end{equation*}
    Then, due to $y_n \to y$ pointwise a.e. in $\Omega$ (at least for a subsequence) and 
    \cref{assu:maxeps}\ref{it:maxeps3} and \ref{it:maxeps4}, there is 
    a subsequence (not relabeled for simplicity) such that 
    \begin{equation*}
        \mx_{\varepsilon_n}'(y_n) \weakly^{*} \chi \quad \text{in } L^\infty(\Omega)
    \end{equation*}
    with 
    \begin{equation*}
        0 \leq \chi \leq 1 \text{ a.e. in }\Omega, \quad \chi = 1 \text{ a.e. in } \{y > 0\},
        \quad \text{and}  \quad \chi = 0 \text{ a.e. in } \{y < 0\}. 
    \end{equation*}
    Note that the transition to a subsequence above is independent of $h$.
    Using \cref{le:PDElemma} and \cref{th:endgame} then yields the second claim.
\end{proof}

\begin{theorem}[optimality system after passing to the limit]\label{thm:kktlimit}
    Let $\bar{u} \in L^2(\Omega)$ be locally optimal for \eqref{eq:p} with associated state 
    $\bar{y} \in Y$. Then there exist a multiplier $\chi \in L^\infty(\Omega)$ 
    and an adjoint state $p\in L^2(\Omega)$ such that 
    \begin{subequations}\label{eq:BKKT}
        \begin{align}
            &p = (G_\chi)^* \partial_y J(\bar{y}, \bar{u}), \label{eq:BKKTad}\\
            &\chi(x) \in \partial_c \max(\bar{y}(x)) \quad \text{a.e. in }\Omega, \label{eq:BKKTsubdiff}\\
            &p + \partial_u J(\bar{y}, \bar{u}) = 0, \label{eq:BKKTgradeq}
        \end{align} 
    \end{subequations}
    where $\partial_c \max : \R \rightrightarrows [0,1]$ denotes the convex subdifferential 
    of the $\max$-function. The solution operator $G_\chi$ is thus an element of $\partial^{sw}_B S(\bar{u})$.
\end{theorem}

\begin{proof}
    Based on the previous results, the proof follows standard arguments, which we briefly sketch for the convenience of the reader.
    We introduce the reduced objective functional associated with \eqref{eq:peps} as 
    \begin{equation*}
        F_\varepsilon : L^2(\Omega) \to \R, \qquad
        F_\varepsilon(u) := J(S_\varepsilon(u), u) + \frac{1}{2}\,\|u - \bar{u}\|_{L^2(\Omega)}^2,
    \end{equation*}
    and consider the following auxiliary optimal control problem
    \begin{equation}\tag{P$_{\varepsilon,r}$}\label{eq:pepsr}
        \left.
            \begin{aligned}
                &\min\nolimits_{u\in L^2(\Omega)} \quad F_\varepsilon(u)\\
                &\text{s.t.}   \quad \|u - \bar{u}\|_{L^2} \leq r,
            \end{aligned}
        \qquad \right\}
    \end{equation}
    where $r>0$ is the radius of local optimality of $\bar{u}$.
    Thanks to the weak continuity of $S_\varepsilon$ by \cref{lem:Sepsfrechet} and 
    the weak lower semi-continuity of $J$ by \cref{assu:standing}, the 
    direct method of the calculus of variations immediately implies the existence of a 
    global minimizer  $\bar{u}_\varepsilon \in L^2(\Omega)$ of \eqref{eq:pepsr}.
    Note that due to the continuous Fr\'echet-differentiability of $J$, 
    the global Lipschitz continuity of $S$, and \eqref{eq:Seps}, 
    there exists (after possibly reducing the radius $r$) a constant $C' > 0$ such that for all sufficiently small $\varepsilon > 0$, it holds
    \begin{equation*}
        | J(S(u),u) -  J(S_\varepsilon(u),u) | \leq C'  \varepsilon \quad \forall\, u \in L^2(\Omega) 
        \text{ with } \|u - \bar{u}\|_{L^2} \leq r.
    \end{equation*}
    As a consequence, we obtain (with the same constant)
    \begin{equation*}
        F_\varepsilon(\bar{u}) = J(S_\varepsilon(\bar{u}),\bar{u})  
        \leq C'  \varepsilon  + J(S(\bar{u}),\bar{u})
    \end{equation*}
    and 
    \begin{align*}
        F_\varepsilon(u) &= J(S_\varepsilon(u),u) +  \frac12\|u - \bar{u}\|_{L^2}^2 \\
                         &\geq  J(S(u),u) +  \frac12 \|u - \bar{u}\|_{L^2}^2 - C'  \varepsilon 
        \quad \forall \, u \in L^2(\Omega) \text{ with } \|u - \bar{u}\|_{L^2} \leq r,
    \end{align*}
    and therefore
    \begin{equation*}
        F_\varepsilon(\bar{u})  < F_\varepsilon(u) 
        \quad \forall \,u \in L^2(\Omega) 
        \text{ with } \sqrt{ 4 C'  \varepsilon} < \|u - \bar{u}\|_{L^2} \leq r.
    \end{equation*}
    Thus, for every $\varepsilon > 0$ sufficiently small, any global solution $\bar{u}_\varepsilon$ 
    of \eqref{eq:pepsr} must necessarily satisfy
    \begin{equation}\label{eq:uoptconv}
        \|\bar{u}_\varepsilon - \bar{u}\|_{L^2} \leq \sqrt{ 4 C'  \varepsilon}.
    \end{equation}
    In particular, for $\varepsilon$ small enough, $\bar{u}_\varepsilon$ 
    is in the interior of the $r$-ball around $\bar{u}$ and, as a global solution of \eqref{eq:pepsr}, 
    is also a local one of \eqref{eq:peps}. It therefore satisfies the first-order necessary 
    optimality conditions of the latter, which, thanks to the chain rule and \cref{lem:Sepsfrechet},
    read
    \begin{equation}
        \big(\partial_y J(S_\varepsilon(\bar{u}_\varepsilon), \bar{u}_\varepsilon), 
        S_\varepsilon'( \bar{u}_\varepsilon)h \big)_{Y}
        + \big( \partial_u J(S_\varepsilon(\bar{u}_\varepsilon), \bar{u}_\varepsilon), 
        h \big )_{L^2} 
        + ( \bar{u}_\varepsilon - \bar{u}  , h )_{L^2} 
        = 0\quad \forall \, h \in L^2(\Omega). \label{eq:necregsys}
    \end{equation}
    From \cref{lemma:regconv} we obtain that there exists a sequence 
    $\varepsilon_n \to 0^+$ and an operator $G \in \partial_{B}^{sw} S(\bar{u})$ such that 
    \begin{equation*}
        S'_{\varepsilon_{n}}(\bar{u}_{\varepsilon_n})h 
        \rightharpoonup G h \text{ in } Y\quad \forall \,h \in L^2(\Omega).
    \end{equation*}
    Further, we deduce from \eqref{eq:uoptconv}, the global Lipschitz continuity of $S$, 
    and \eqref{eq:Seps}, that $S_{\varepsilon_{n}}(\bar{u}_{\varepsilon_n})  \to S(\bar{u})$ 
    in $Y$. Combining all of the above and using our assumptions on $J$, we can pass to the limit 
    $\varepsilon_n \to 0$ in \eqref{eq:necregsys} to obtain 
    \begin{equation*}
        \big(\partial_y J(S (\bar{u} ), \bar{u} ), G h \big )_{Y}  
        + \big ( \partial_u J(S(\bar{u} ), \bar{u} ), h \big )_{L^2} 
        =  0 \quad   \forall\, h \in L^2(\Omega).
    \end{equation*}
    By setting $p := G^* \partial_y J(S (\bar{u} ), \bar{u} )$, this together with  
    \cref{th:endgame} and \cref{rem:subdiffmax2} finally proves the claim.
\end{proof}

\begin{corollary}\label{cor:adPDE}
    Assume that $J$ is continuously Fr\'echet-differentiable from $H^1_0(\Omega) \times L^2(\Omega)$ to $\R$.
    If $\bar{u} \in L^2(\Omega)$ is locally optimal for \eqref{eq:p} with associated state 
    $\bar{y}$, then there exists a multiplier $\chi\in L^\infty(\Omega)$ and an adjoint state 
    $p\in H^1_0(\Omega)$ such that 
    \begin{subequations}
        \begin{align}
            &-\Delta p + \chi\,p = \partial_y J(\bar{y}, \bar{u}), \label{eq:adPDE}\\
            &\chi(x) \in \partial_c \max(\bar y(x)) \quad \text{a.e. in }\Omega,\\
            &p + \partial_u J(\bar{y}, \bar{u}) = 0.
        \end{align} 
    \end{subequations}
    If $J$ is even Fr\'echet-differentiable from $L^2(\Omega) \times L^2(\Omega)$ to $\R$, then $p \in Y$.
\end{corollary}

\begin{proof}
    According to \cref{def:Gchi}, $G_\chi$ is the solution operator of 
    \eqref{eq:linpdechi}, which is formally self-adjoint. Thus we can argue as in 
    \cite[Sec.~4.6]{Troeltzsch} to deduce \eqref{eq:adPDE} and the $H^1$-regularity of $p$. 
    The $Y$-regularity is again obtained by bootstrapping.
\end{proof}

\subsection{Interpretation of the optimality conditions in the limit}\label{sec:Bstat}

In classical non-smooth optimization, optimality conditions of the form $0\in \partial_* f(x)$, where  $\partial_* f$ denotes one of the various subdifferentials of $f$,
frequently appear when a function $f: X \to \R$ is minimized over a normed linear space $X$;  
we only refer to \cite[Secs.~7 and 9]{schirotzek} and the references therein.
With the help of the results of \cref{sec:bouli} (in particular \cref{th:endgame}), we
are now in the position to interpret the optimality system in \eqref{eq:BKKT} in this spirit. 
To this end, we first consider the reduced objective and establish the following result 
for its Bouligand subdifferential.
\begin{proposition}[chain rule]\label{prop:chainrule}
    Let $u \in L^2(\Omega)$ be arbitrary but fixed
    and let $F : L^2(\Omega) \to \R$ be the reduced objective for \eqref{eq:p} defined by $F(u) := J(S(u), u)$. 
    Moreover, set $y:=S(u)$. Then it holds
    \begin{equation*}
        \begin{aligned}
            &\left\{ G^*\partial_y J(y, u) + \partial_u J( y, u)  :  G \in  \partial_B^{sw} S(u)\right\} \\
            &\qquad\qquad \subseteq  \partial_B F(u) := \{ w \in L^2(\Omega):\, 
            \begin{aligned}[t]
                & \text{there exists } (u_n) \subset L^2(\Omega) \text{ with } u_n \to u\text{ in } L^2(\Omega) \\
                & \text{such that } F \text{ is G\^ateaux in } u_n \text{ for all } n \in \N \\
                & \text{and }F'(u_n) \rightharpoonup w \text{ in } L^2(\Omega) \text{ as }n\to \infty\}.
            \end{aligned}    
        \end{aligned}
    \end{equation*}
\end{proposition}

\begin{proof}
    Let $u \in L^2(\Omega)$ and $G \in  \partial_B^{sw} S(u)$ be arbitrary but fixed. 
    Then the definition of $\partial_B^{sw} S(u)$ guarantees the existence of a sequence 
    $u_n \in L^2(\Omega)$ of G\^ateaux points with $u_n \to u$ in $L^2(\Omega)$ 
    and $S'(u_n)h \rightharpoonup G h$ in $Y$ for all $h \in L^2(\Omega)$. 
    Since $J$ is Fr\'echet- and thus Hadamard-differentiable and so is $S$ by 
    \cref{thm:rabl}, we may employ the chain rule 
    to deduce that $F$ is G\^ateaux-differentiable in the points $u_n \in L^2(\Omega)$ with
    derivative
    \begin{equation*}
        F'(u_n) = S'(u_n)^*\partial_y J(y_n, u_n) + \partial_u J( y_n, u_n) \in L^2(\Omega)
    \end{equation*}
    for $y_n := S(u_n)$. As $y_n \to y$ in $Y$ by \cref{prop:basic} 
    and $J : Y \times L^2(\Omega) \to \R$ is continuously Fr\'echet-differentiable
    by \cref{assu:standing}, we obtain for every $h\in L^2(\Omega)$ that 
    \begin{equation*}
        \begin{aligned}
            \big( F'(u_n), h \big)_{L^2} 
            &=  \big( \partial_y J(y_n, u_n) , S'(u_n)h \big)_{Y} 
            + \big(\partial_u J( y_n, u_n) , h\big)_{L^2}\qquad\qquad \\
            &\to \big( \partial_y J(y, u) , G h\big )_{Y} +  \big(\partial_u J( y, u) , h \big)_{L^2}. 
        \end{aligned}
    \end{equation*}
    Since $h\in L^2(\Omega)$ was arbitrary, 
    this proves $G^*\partial_y J(y, u) + \partial_u J( y, u) \in \partial_B F(u)$.
\end{proof}

With the above result, we can now relate the optimality conditions obtained via regularization 
to the Bouligand subdifferential of the reduced objective, 
and in this way rate the strength of the optimality system in \eqref{eq:BKKT}.
\begin{theorem}[limit optimality system implies Bouligand-stationarity]\label{theorem:summary}
    It holds:
    \begin{gather*}
        \bar{u}   \text{ is locally optimal for }\eqref{eq:p}\\
        \Downarrow\\
        \text{there exist $\chi\in L^\infty(\Omega)$ and $p\in L^2(\Omega)$ such that \eqref{eq:BKKT} holds}\\
        \Downarrow\\
        \bar{u} \text{ is Bouligand-stationary for }\eqref{eq:p} 
        \text{ in the sense that } 0 \in \partial_B F(u)\\
        \Downarrow\\
        \bar{u} \text{ is Clarke-stationary for }\eqref{eq:p} 
        \text{ in the sense that } 0 \in \partial_C F(u)
    \end{gather*}
    Here, $\partial_C F(u)$ denotes the Clarke subdifferential as defined in \cite[Sec.~2.1]{Clarke:1990}.
\end{theorem}

\begin{proof}
    The first two implications immediately follow from \cref{thm:kktlimit} and 
    \cref{prop:chainrule}. For the third implication, observe that  
    the weak closedness of $\partial_C F(u)$ (see \cite[Prop.~2.1.5b]{Clarke:1990}) and $F'(u) \in \partial_C F(u)$ 
    in all G\^ateaux points (cf.~\cite[Prop.~2.2.2]{Clarke:1990}) result in 
    $\partial_B F(u) \subseteq \partial_C F(u)$. 
\end{proof}

\begin{remark}
    The above theorem is remarkable for several reasons:
    \begin{itemize}
        \item[(i)] \cref{theorem:summary} shows that  $0 \in \partial_B F(u)$ 
            is a necessary optimality condition for the optimal control problem \eqref{eq:p}. 
            This is in general not true even in finite dimensions, as the minimization 
            of the absolute value function shows.
        \item[(ii)] The above shows that the necessary optimality condition in 
            \cref{thm:kktlimit}, which is obtained by regularization, is comparatively strong. 
            It is stronger than Clarke-stationarity and even stronger than Bouligand-stationarity
            (which is so strong that it does not even make sense in the majority of problems).
    \end{itemize}
\end{remark}

\begin{remark}\label{rem:controlconstr}
    It is easily seen that the limit analysis in \cref{sec:limit} readily carries 
    over to control constrained problems involving an additional constraint of the form  
    $u\in U_{\textup{ad}}$ for a closed and convex $U_{\textup{ad}}\subset L^2(\Omega)$. 
    The optimality system arising in this way is identical to \eqref{eq:BKKT} except 
    for \eqref{eq:BKKTgradeq}, which is replaced by the variational inequality
    \begin{equation*}
        (p + \partial_u J(\bar{y}, \bar{u}), u - \bar{u}) \geq 0 
        \quad \forall\,u \in U_{\textup{ad}}.
    \end{equation*}
    The interpretation of the optimality system arising in this way in the spirit 
    of \cref{theorem:summary} is, however, all but straightforward, as 
    it is not even clear how to define the Bouligand subdifferential of the reduced objective 
    in the presence of control constraints. Intuitively, one would choose the 
    approximating sequences in the definition of $\partial_B F$ from the feasible set
    $U_{\textup{ad}}$, but then the arising subdifferential could well be empty. 
    This gives rise to future research.
\end{remark}

\subsection{Strong stationarity}

Although comparatively strong, the optimality conditions in \cref{thm:kktlimit} are not the most rigorous 
ones, as we will see in the sequel. To this end, we apply a method of proof which was developed in \cite{ms16} 
for optimal control problems governed by non-smooth semilinear parabolic PDEs and inspired by the analysis in \cite{m76, mp84}.
We begin with an optimality condition without dual variables.
\begin{proposition}[purely primal optimality conditions]\label{prop:primal}
    Let $\bar{u} \in L^2(\Omega)$ be locally optimal for \eqref{eq:p} 
    with associated state $\bar{y} = S(\bar{u}) \in Y$. Then there holds
    \begin{equation}\label{eq:VI}
        F'(\bar{u}; h) = \partial_y J(\bar{y}, \bar{u})S'(\bar{u};h) 
        + \partial_u J(\bar{y}, \bar{u})h \geq 0  \quad \forall\, h\in L^2(\Omega).
    \end{equation}
\end{proposition}

\begin{proof}
    As already argued above, $J: Y \times L^2(\Omega) \to \R$ and $S: L^2(\Omega) \to Y$ 
    are Hadamard-differentiable, so that the reduced objective $F: L^2(\Omega) \ni u \mapsto J(S(u), u) \in \R$ 
    is Hadamard-differentiable with directional derivative 
    $F'(u;h) = \partial_y J(S(u),u)S'(u;h) + \partial_u J(S(u),u) h$ by the chain rule for Hadamard-differentiable mappings. 
    Thus by classical arguments, the local optimality of $\bar{u}$ implies $F'(\bar u; h) \geq 0$ for all $h\in L^2(\Omega)$.
\end{proof}

\begin{lemma}\label{lem:adjoint}
    Let $p\in L^2(\Omega)$ fulfill \eqref{eq:BKKTad}, i.e., $p = (G_\chi)^*\partial_y J(\bar{y}, \bar{u})$ 
    with some $\chi \in L^\infty(\Omega)$, $\chi \geq 0$. Then for every $v\in Y$ there holds
    \begin{equation}\label{eq:pdefoo2}
        (-\Delta v + \chi v , p)_{L^2(\Omega)}
        = \dual{\partial_y J(\bar{y}, \bar{u})}{v}_{Y', Y}.
    \end{equation} 
\end{lemma}

\begin{proof}
    Let $v\in Y$ be arbitrary and define $g\in L^2(\Omega)$ by $g := -\Delta v + \chi\, v$ so that $v = G_\chi g$.
    Then $p = (G_\chi)^* \partial_yJ(\bar{y}, \bar{u})$ implies
    \begin{equation*}
        (-\Delta v + \chi v , p)_{L^2(\Omega)}
        = (g,p)_{L^2(\Omega)} = \dual{\partial_y J(\bar{y}, \bar{u})}{G_\chi g}_{Y',Y} 
        = \dual{\partial_y J(\bar{y}, \bar{u})}{v}_{Y', Y}
    \end{equation*}
    as claimed.
\end{proof}

\begin{theorem}[strong stationarity]\label{thm:strongstat}
    Let $\bar{u} \in L^2(\Omega)$ be locally optimal for \eqref{eq:p} with associated state 
    $\bar{y} \in Y$. Then there exist a multiplier $\chi \in L^\infty(\Omega)$ 
    and an adjoint state $p\in L^2(\Omega)$ such that 
    \begin{subequations}\label{eq:SKKT}
        \begin{align}
            &p = (G_\chi)^* \partial_y J(\bar{y}, \bar{u}), \label{eq:SKKTad}\\
            &\chi(x) \in \partial_c \max(\bar y(x)) \quad \text{a.e. in }\Omega, \label{eq:SKKTsubdiff}\\
            &p(x) \leq 0 \quad\text{a.e. in } \{\bar y=0\}, \label{eq:SKKTpsign}\\
            &p + \partial_u J(\bar{y}, \bar{u}) = 0. \label{eq:SKKTgradeq}
        \end{align} 
    \end{subequations}
\end{theorem}

\begin{proof}
    From \cref{thm:kktlimit}, we know that there exist $p\in L^2(\Omega)$ and $\chi\in L^\infty(\Omega)$ 
    such that \eqref{eq:BKKT} is valid, which already gives \eqref{eq:SKKTad}, \eqref{eq:SKKTsubdiff}, and \eqref{eq:SKKTgradeq}. 
    It remains to show \eqref{eq:SKKTpsign}. To this end, let $v\in Y$ be arbitrary and define 
    \begin{equation}\label{eq:rablpde}
        h := -\Delta v + \mathbb{1}_{\{\bar{y}=0\}} \max(0,v) + \mathbb{1}_{\{\bar{y}>0\}} v 
        = -\Delta v + \mx'(\bar{y};v) \in L^2(\Omega)
    \end{equation}
    so that $v = S'(\bar{u};h)$ by \cref{thm:rabl}. 
    By testing \eqref{eq:rablpde} with $p$ and using \eqref{eq:BKKTgradeq}
    and \eqref{eq:VI} from \cref{prop:primal}, we arrive at
    \begin{equation}\label{eq:pdefoo1}
        \begin{aligned}[b]
            (-\Delta v + \mx'(\bar{y};v) ,  p)_{L^2(\Omega)}
            & = (h, p)_{L^2(\Omega)} \\
            & = (-\partial_u J(\bar{y}, \bar{u}), h)_{L^2(\Omega)}\\
            & \leq \dual{\partial_y J(\bar{y}, \bar{u})}{S'(\bar{u};h)}_{Y', Y} 
            = \dual{\partial_y J(\bar{y}, \bar{u})}{v}_{Y', Y}.
        \end{aligned}
    \end{equation}
    On the other hand, we know from \cref{lem:adjoint} that $p$ and $\chi$ satisfy \eqref{eq:pdefoo2}.
    Subtracting this equation from \eqref{eq:pdefoo1} and using the density of $Y \embed L^2(\Omega)$ 
    and the global Lipschitz continuity of $L^2(\Omega) \ni v \mapsto \mx'(\bar{y};v) \in L^2(\Omega)$ yields
    \begin{equation}\label{eq:L2ineq}
        \int_\Omega \big( \mx'(\bar{y}; v) - \chi\, v \big) p\,\d x \leq 0 
        \quad \forall\, v \in L^2(\Omega).
    \end{equation}  
    Note that due to \eqref{eq:SKKTsubdiff}, the bracket in \eqref{eq:L2ineq} vanishes a.e.\ in $\{\bar{y}\neq 0\}$. 
    Thus, we obtain 
    \begin{equation*}
        \int_{\{\bar y = 0\}} \big( \max(0, v) - \chi\, v \big) p\,\d x = \int_{\{\bar y = 0\}} \Big (  ( 1 - \chi\,)\max(0, v) + \chi \max(0, -v) \Big )  p\,\d x \leq 0 
    \end{equation*}
    for all $v \in L^2(\Omega)$. The above implies that $ ( 1 - \chi\,)v p \leq 0$ and $\chi v p \leq 0$ holds a.e.\ in $\{\bar y = 0\}$ for all $0 \leq v \in L^2(\Omega)$, and this in turn yields by addition that $v p \leq 0$ holds a.e.\ in $\{\bar y = 0\}$ for all $0 \leq v \in L^2(\Omega)$. Inequality \eqref{eq:SKKTpsign} now follows immediately. 
\end{proof}

\begin{proposition}\label{prop:equiv}
    The strong stationarity conditions are equivalent to the purely primal 
    optimality conditions, i.e., $\bar{u}\in L^2(\Omega)$ together with its state 
    $\bar{y}$ and a multiplier $\chi$ and an adjoint state $p$ satisfies \eqref{eq:SKKT} if and only if 
    they also fulfill the variational inequality \eqref{eq:VI}.
\end{proposition}

\begin{proof}
    Let $h \in L^2(\Omega)$ be arbitrary and define $\delta = S'(\bar{u};h)$. Then 
    the gradient equation in \eqref{eq:SKKTgradeq} and \cref{thm:rabl} give 
    \begin{equation}\label{eq:gradeq}
        \begin{aligned}[b]
            (- \partial_u J(\bar{y}, \bar{u}), h)_{L^2} 
            = (p, h)_{L^2} 
            & = (-\Delta \delta+ \mx'(\bar{y}; \delta), p )_{L^2} \\
            & = (-\Delta \delta + \chi\,\delta , p)_{L^2} 
            + (\mx'(\bar{y}; \delta) - \chi\,\delta, p )_{L^2}.
        \end{aligned}
    \end{equation}
    For the last term, \eqref{eq:SKKTsubdiff} and the sign condition in \eqref{eq:SKKTpsign} yield
    \begin{equation*}
        \begin{aligned}
            (\mx'(\bar{y}; \delta) - \chi\,\delta, p )_{L^2}
            =\int_{\{y=0\}\cap\{\delta \geq 0\}} (1 - \chi) \delta \,p \, \d x
            + \int_{\{y=0\}\cap\{\delta < 0\}} (- \chi) \delta \,p \, \d x \leq 0.
        \end{aligned}
    \end{equation*}
    Together with \cref{lem:adjoint}, this implies that \eqref{eq:gradeq} results in
    \begin{equation*}
        (- \partial_u J(\bar{y}, \bar{u}), h)_{L^2} 
        \leq \dual{\partial_y J(\bar{y}, \bar{u})}{\delta}_{Y', Y},
    \end{equation*}
    which is \eqref{eq:VI}.

    The other direction follows analogously to \cite[proof of Thm.~5.3]{ms16}. 
    Assume that $\bar u \in L^2(\Omega)$ with $\bar y := S(\bar u)$ satisfies \eqref{eq:VI}.
    Defining
    \begin{equation}\label{eq:subdiffSKKT1}
        p := - \partial_u J(\bar{y}, \bar{u}) \in L^2(\Omega)
    \end{equation}
    and $\eta := S'(\bar u; h)$, we have 
    \begin{equation*}
        \int_\Omega \Big ( -\Delta \eta + \mathrm{max}'(\bar y; \eta) \Big ) p\,\mathrm{d}x = \int_\Omega h p \,\mathrm{d}x\qquad \forall h\in L^2(\Omega).
    \end{equation*}
    From \eqref{eq:VI} we thus obtain that
    \begin{equation*}
        \int_\Omega  \mathrm{max}'(\bar y; \eta) p\,\mathrm{d}x \leq \left \langle  \Delta p + \partial_y J(\bar{y}, \bar{u}), \eta \right \rangle_{Y', Y}.
    \end{equation*}
    Let now $\lambda := \Delta p + \partial_y J(\bar{y}, \bar{u}) \in Y'$. The surjectivity of $S'(\bar u; \cdot) : L^2(\Omega) \to Y$ then yields
    \begin{equation*}
        \left \langle  \lambda, \eta \right \rangle_{Y', Y} \leq \int_\Omega  -\mathrm{max}'(\bar y; -\eta)   p\,\mathrm{d}x = \int_\Omega   \mathbb{1}_{\{\bar y > 0\}}\eta p 
        + \mathbb{1}_{\{\bar y = 0\}}\min(0, \eta)p \, \mathrm{d}x \qquad \forall \eta \in Y.
    \end{equation*}
    By the Hahn--Banach extension theorem, we can extend $\lambda$ from $Y$ to $L^2(\Omega)$, i.e., we obtain a $\mu \in L^2(\Omega)$ that satisfies this variational inequality for all $\eta\in L^2(\Omega)$. Pointwise inspection then shows that
    \begin{equation}\label{eq:subdiffSKKT2}
        \mu = p \text{ a.e.~in } \{\bar y > 0\},\qquad \mu = 0 \text{ a.e.~in } \{\bar y < 0\},\qquad p \leq \mu \leq 0 \text{ a.e.~in } \{\bar y = 0\}.
    \end{equation}
    Hence,
    \begin{equation*}
        \chi :=
        \begin{cases}
            1 & \text{a.e.\  in } \{\bar y > 0\},\\
            \frac{\mu}{p}\quad & \text{a.e.\ in } \{\bar y = 0\} \cap \{p \neq 0\},\\
            0 &  \text{a.e.\ } \text{else},
        \end{cases}
    \end{equation*}
    satisfies $\chi\in\partial_c \max(\bar y)$ a.e. as well as 
    \begin{equation*}
       \chi p = \mu = \Delta p + \partial_y J(\bar{y}, \bar{u}).
    \end{equation*}
    Together with \eqref{eq:subdiffSKKT1} and the third relation of \eqref{eq:subdiffSKKT2}, this yields \eqref{eq:SKKT}.
\end{proof}

\begin{remark}
    As in case of the optimality system \eqref{eq:BKKT}, the regularity of the adjoint state 
    in \cref{thm:strongstat} is again only limited by the mapping 
    and differentiability properties of the objective functional. 
    Thus, arguing as in \cref{cor:adPDE}, one shows that if $J$ is differentiable from $H^1_0(\Omega)\times L^2(\Omega)$ 
    or $L^2(\Omega)\times L^2(\Omega)$ to $\R$, the adjoint state
    $p$ satisfying \eqref{eq:adPDE} is an element of 
    $H^1_0(\Omega)$ or $Y$, respectively.
\end{remark}

\begin{remark}
    Although the optimality system \eqref{eq:BKKT} is comparatively strong by 
    \cref{theorem:summary}, it provides less information compared to the strong 
    stationarity conditions in \eqref{eq:SKKT} since it lacks the sign condition \eqref{eq:SKKTpsign} for the adjoint state. The conditions \eqref{eq:SKKT} can be seen as the most rigorous 
    dual-multiplier based optimality conditions, as by \cref{prop:equiv} they are equivalent to the purely primal condition. We point out, however, that the method of proof of  
    \cref{thm:strongstat} can in general not be transferred to the case with 
    additional control constraints (e.g., $u\in U_{\textup{ad}}$ 
    for a closed and convex set $U_{\textup{ad}}$), since it requires the set 
    $\{S'(\bar{u};h) : h \in \operatorname{cone}(U_\textup{ad} - \bar{u})\}$ to be dense in 
    $L^2(\Omega)$. In contrast to this, the adaptation of the limit analysis 
    in \cref{sec:limit} to the case with additional control constraints is
    straightforward as mentioned in \cref{rem:controlconstr}. 
\end{remark}

\section{Algorithms and numerical experiments}\label{sec:numerics}

One particular advantage of the optimality system in \eqref{eq:BKKT} is that it seems amenable to numerical solution 
as we will demonstrate in the following. 
We point out, however, that we do not present a comprehensive convergence analysis for 
our algorithm to compute stationary points satisfying \eqref{eq:BKKT} but only 
a feasibility study.
For the sake of presentation, we consider here an $L^2$ tracking objective of the form
\begin{equation}\label{eq:tracking}
    J(y,u) := \frac{1}{2}\, \|y - y_d\|_{L^2(\Omega)}^2 + \frac{\alpha}{2}\,\|u\|_{L^2(\Omega)}^2
\end{equation}
with a given desired state $y_d \in L^2(\Omega)$ and a Tikhonov parameter $\alpha > 0$.

\subsection{Discretization and semi-smooth Newton method}

Let us start with a short description of our discretization scheme, where we restrict ourselves from now on to the case $\Omega\subset\R^2$.
For the discretization of the state and the control variable, we use standard continuous 
piecewise linear finite elements (FE), cf., e.g., \cite{bs94}. 
Let us denote by $V_h \subset H^1_0(\Omega)$ the associated FE space spanned by the 
standard nodal basis functions $\varphi_1, \dots, \varphi_n$. 
The nodes of the underlying triangulation $\TT_h$ 
belonging to the interior of the domain $\Omega$
are denoted by $x_1, \dots, x_n$.
We then discretize the state equation in \eqref{eq:p} by employing a 
mass lumping scheme for the non-smooth nonlinearity. Specifically,
we consider the discrete state equation
\begin{equation}\label{eq:statediscr}
    \int_\Omega \nabla y_h\cdot \nabla v_h \,\d x 
    + \sum_{T\in \TT_h} \frac{1}{3}\,|T| \sum_{x_i \in \overline{T}} \max(0, y_h(x_i))\, v_h(x_i)
    = \int_\Omega u_h\,v_h\,\d x \quad \forall\, v_h \in V_h,
\end{equation}
where $y_h, u_h\in V_h$ denote the FE-approximations of $y$ and $u$. 
With a slight abuse of notation, we from now on denote the coefficient vectors $(y_h(x_i))_{i=1}^n$
and $(u_h(x_i))_{i=1}^n$ by $y, u \in \R^n$. The discrete state equation can then be written  
as the nonlinear algebraic equation
\begin{equation}\label{eq:stateFE}
    A y + D \max(0,y) = M u,
\end{equation}
where $A := (\int_\Omega \nabla \varphi_i\cdot \nabla \varphi_j\,\d x)_{ij=1}^n \in \R^{n\times n}$
and $M := (\int_\Omega \varphi_i\, \varphi_j\,\d x)_{ij=1}^n \in \R^{n\times n}$ denote 
stiffness and mass matrix, $\max(0,.): \R^n \to \R^n$ is the componentwise $\max$-function, and 
\begin{equation*}
    D := \diag\Big(\int_\Omega\varphi_i(x)\,\d x \Big) = \diag\Big(\frac{1}{3}\,  \omega_i \Big) \in \R^{n\times n}
\end{equation*}
with $\omega_i = |\supp(\varphi_i)|$ denoting the lumped mass matrix. Due to the monotonicity of  the $\max$-operator, one easily shows that 
\eqref{eq:statediscr} and \eqref{eq:stateFE} admit a unique solution for every 
control vector $u$. The objective functional is discretized by means of a suitable interpolation operator $I_h$
(e.g., the Cl\'ement interpolator or, if $y_d\in C(\overline{\Omega})$, the Lagrange interpolator).
If -- again by the abuse of notation -- we denote the coefficient vector of $I_h y_d$ with respect to the nodal basis 
by $y_d$, we end up with the discretized objective 
\begin{equation*}
    J_h : \R^n \times \R^n \to \R, \quad 
    J_h(y,u) := \frac{1}{2}\, (y - y_d)^\top M (y - y_d) + \frac{\alpha}{2}\, u^\top M u.
\end{equation*}
Again regularizing the $\max$-function in \eqref{eq:stateFE}, a limit analysis analogous to 
\cref{sec:limit} yields the following discrete counterpart to \eqref{eq:BKKT} with 
vectors $p, \chi\in \R^n$ as necessary optimality conditions for the discretized optimal control problem:
\begin{subequations}\label{eq:BKKTFE}
    \begin{align}
        &A y + D \max(0,y) = - \frac{1}{\alpha} M p, \label{eq:BKKTFE1}\\
        &A p + D\,\chi \circ p = M (y - y_d), \label{eq:BKKTFE2}\\
        &\chi_i \in \partial_c \max(y_i), \quad i = 1, \dots, n. \label{eq:BKKTFE3}
    \end{align}
\end{subequations}
Here, $a\circ b := (a_i b_i)_{i=1}^n$ denotes the Hadamard product, and we have eliminated 
the control by means of the gradient equation $p + \alpha u = 0$.

Next, we reformulate \eqref{eq:BKKTFE} as a non-smooth system of equations. 
To this end, let us introduce for a given $\gamma > 0$ the proximal point mapping 
$\prox_\gamma: \R\to\R$ of $\max$ by 
\begin{equation}\label{eq:proxgam}
    \prox_\gamma(x) := \argmin_{s\in\R} \Big( \max(0,s) + \frac{1}{2\gamma} \, |s-x|^2 \Big)
    = \begin{cases}
        x, & x < 0,\\
        0, & x\in [0,\gamma], \\
        x - \gamma, & x > \gamma.
    \end{cases}
\end{equation}
Since the proximal point mapping of $\max$ coincides with the resolvent $(I+\gamma\partial_c\max)^{-1}$ of its convex subdifferential, it is straightforward to show that $g\in \partial_c \max(z)$ if and only if 
$z = \prox_\gamma(z + \gamma \,g)$. Thus, for every $\gamma > 0$, 
\eqref{eq:BKKTFE3} is equivalent to the non-smooth equation 
\begin{equation}\tag{\ref{eq:BKKTFE3}$'$}\label{eq:BKKTFE3prime}
    y_i = \prox_\gamma(y_i + \gamma\,\chi_i), \quad i = 1, \dots, n.
\end{equation}
Since $\prox_\gamma$ is Lipschitz continuous and piecewise continuously differentiable by \eqref{eq:proxgam}, 
it is semi-smooth as a function from $\R\to \R$. As the same holds for the $\max$-function, 
it seems reasonable to apply a semi-smooth Newton method to numerically solve
the system consisting of \eqref{eq:BKKTFE1}, \eqref{eq:BKKTFE2}, and \eqref{eq:BKKTFE3prime}.
However, the application of a Newton-like scheme to \eqref{eq:BKKTFE} is 
a delicate issue. This can already be seen by observing that $\chi_i$ is not unique in points where $y_i$ 
and $p_i$ vanish at the same time. Moreover, the Newton matrix may well be singular. 
For a clearer notation, let us introduce the index sets $\II_+ := \{i : y_i > 0\}$
and $\II_\gamma := \{i: y_i + \gamma \, \chi_i \notin [0,\gamma]\}$, and denote by 
$\mathbb{1}_{\II_+}$ and $\mathbb{1}_{\II_\gamma}$ the characteristic vectors of these index sets. 
The Newton matrix associated with \eqref{eq:BKKTFE} is given by
\begin{equation*}
    \begin{pmatrix}
        A + D\, \diag(\mathbb{1}_{\II_+}) & \frac{1}{\alpha}\,M & 0 \\
        -M & A + D\, \diag(\chi) & D\,\diag(p) \\
        D - D\diag(\mathbb{1}_{\II_\gamma}) & 0 & - \gamma\,D\diag(\mathbb{1}_{\II_\gamma})
    \end{pmatrix},
\end{equation*}
where we have multiplied the last row corresponding to \eqref{eq:BKKTFE3prime} by the lumped mass matrix to ensure a uniform scaling of the Newton equations.
As already noted, the matrix becomes singular if there is an index $i\in \{1, \dots , n\}$ such that 
$p_i =  0$ and $y_i + \gamma \, \chi_i \in [0,\gamma]$. 
To resolve this problem, we follow an active set-type strategy where we remove the components of $\chi$ corresponding to these indices from the Newton equation and leave them unchanged in the Newton update. This is motivated by the fact that components that are fixed in this way will likely either be changed in a further iteration (and hence no longer be singular) or correspond to underdetermined components in the optimality conditions for which we thus take the (feasible) starting value as a final value. A similar strategy was also followed in \cite{Pieper} (see Chap.~3.4.1 and Rem.~3.9\,ii)). This procedure worked well in our numerical tests at least for small values of $\gamma$ (where the influence of the critical set on the Newton matrix is smaller), as we will demonstrate below.

\begin{remark}\label{rem:strongstatnum}
    It is not clear how to numerically solve the strong stationarity conditions in \eqref{eq:SKKT} 
    since the system may well become overdetermined by the sign condition in \eqref{eq:SKKTpsign}. 
\end{remark}

\subsection{Numerical results}

We now present two different examples with a constructed exact solution to \eqref{eq:BKKT} in order to demonstrate convergence of the proposed algorithm and to illustrate the dependence on the parameters $\alpha$ and $\gamma$ as well as on the mesh size $h$.
In both examples, the state vanishes in parts of the domain so that the non-smoothness of the $\max$-function becomes apparent. 
For this purpose, we introduce an additional inhomogeneity in the state equation, i.e., we replace the PDE 
in \eqref{eq:p} by
\begin{equation*}
    y \in H_0^1(\Omega), \quad -\Delta y + \max(0, y) = u + f \; \text{ in }\Omega
\end{equation*}
with a given function $f\in L^2(\Omega)$. It is easy to see that this modification does not 
influence the analysis in the preceding sections. 
The domain is chosen as the unit square $\Omega = [0,1]^2\subset \R^2$, which is discretized by means of Friedrich--Keller triangulations.
In all cases, we take as a starting guess for the Newton iteration $y^0=p^0=\chi^0=0$, and terminate the iteration if either the combined norm of the residuals in \eqref{eq:BKKTFE1}, \eqref{eq:BKKTFE2}, and \eqref{eq:BKKTFE3prime} becomes less than $10^{-12}$ or if the maximum number of $25$ Newton iterations is reached. The Newton system in each iteration is solved by \textsc{matlab}'s sparse direct solver.

\bigskip

In the first example, the optimal state and adjoint state are set to 
\begin{equation*}
    y(x_1, x_2) = \sin(\pi\, x_1) \sin(2\pi\, x_2) \quad \text{ and }\quad
    p \equiv 0,
\end{equation*}
and the data $f$ and $y_d$ are constructed such that \eqref{eq:BKKT} is fulfilled, i.e., the optimal control is $u\equiv 0$.
We note that there is a subset where $y$ and $p$ vanish at the same time, but 
it is only of measure zero. 
\Cref{tab:ex1} presents the numerical results for different values of the mesh size $h$,
the Tikhonov parameter $\alpha$ in the objective in \eqref{eq:tracking}, and the parameter 
$\gamma$ in the proximal point mapping.
For the state $y$, we report the relative error of the computed approximation $y_h$ with respect to the ($L^2$ projection of the) constructed optimal state $y$ in the continuous $L^2$ norm. For this choice of the adjoint state $p$, the relative error is of course not appropriate, and we thus report here the absolute error in the continuous $L^2$ norm. The error for $\chi$ is given in the discrete
$L^\infty$ norm, $\|\chi_h - \chi\|_{L^\infty,h} := \max_{i = 1, ..., n} |\chi_h(x_i) - \chi(x_i)|$, where the
$x_i$ are the interior nodes of the triangulation. (Note that $\chi$ is everywhere constant except in
$\{ y = 0 \}$, and the triangulation is chosen such that $y(x_i)\neq 0$.)

\begin{table}
    \caption{Numerical results in the first example}\label{tab:ex1}
    \centering
    \begin{tabular}{ccccccc}
        \toprule
        $h$ & $\alpha$ & $\gamma$ & $\frac{\|y_h - y\|_{L^2}}{\|y\|_{L^2}}$
            & $\|p_h - p\|_{L^2}$
            & $\|\chi_h - \chi\|_{L^\infty,h}$
            & \#~Newton\\
        \midrule
        \num{0.030303030303030} & \num{0.0001} & \num{0.0001}
                                & \num{0.001152239343045} & 
        \num{0.00001035567491004718}  & \num{0.0000008150131505891369} 
                                      & 3 \\ 
        \num{0.015384615384615} & \num{0.0001} & \num{0.0001}
                                & \num{0.0002962149528697947} & 
        \num{0.000002679405106924650}  & \num{0.0000008148769177172133} 
                                       & 3 \\ 
        \num{0.007751937984496} & \num{0.0001} & \num{0.0001}
                                & \num{0.00007515492583947963} & 
        \num{0.0000006809496128184699}  & \num{0.0000008155578545099859} 
                                        & 3 \\ 
        \num{0.003891050583658} & \num{0.0001} & \num{0.0001}
                                & \num{0.00001893193153042300} &  
        \num{0.0000001716079487056943}  & \num{0.0000008156259169223314} 
                                        & 3 \\ 
        \midrule
        \num{0.007751937984496} & \num{0.0001} & \num{0.01}
                                & -- & -- & -- & no conv. \\ 
        \num{0.007751937984496} & \num{0.0001} & \num{0.001}
                                & -- & -- & -- & no conv. \\ 
        \num{0.007751937984496} & \num{0.0001} & \num{0.00001}
                                & \num{0.00007515492583949867} & 
        \num{0.0000006809496128183225}  & \num{0.000003177810190947213} 
                                        & 3 \\ 
        \num{0.007751937984496} & \num{0.0001} & \num{0.000001}
                                & \num{0.00007515492583953322} & 
        \num{0.0000006809496128190883}  & \num{0.000009178413752895815} 
                                        & 3 \\ 
        \midrule
        \num{0.007751937984496} & \num{0.01} & \num{0.0001}
                                & \num{0.0003266863213013805} & 
        \num{0.000003240867239796822}  & \num{0.0000008154217126419962} 
                                       & 3 \\ 
        \num{0.007751937984496} & \num{0.001} & \num{0.0001}
                                & \num{0.0002443542586863312} & 
        \num{0.000002405294957718857}  & \num{0.0000008153536331835063} 
                                       & 3 \\ 
        \num{0.007751937984496} & \num{0.000001} & \num{0.0001}
                                & \num{0.000002449250392773870} & 
        \num{0.000000009204169223006045}  & \num{0.0000008148769177172133} 
                                          & 3 \\ 
        \num{0.007751937984496} & \num{0.00000001} & \num{0.0001}
                                & \num{0.0000001198993334906963} & 
        \num{0.00000000009452339609744186}  & \num{0.0000008152855480401214} 
                                            & 3 \\ 
        \bottomrule
    \end{tabular}
\end{table}

First, we remark that for almost all combinations of parameter values, only a few Newton iterations are needed to reach a residual norm below  $10^{-12}$. Regarding the dependence on the mesh size, we can observe quadratic convergence of the state $y$ and the adjoint state $p$. Since in this case, the active set satisfies $\chi(x_i)\in\{0,1\}$ for all $x_i$, the approximation only depends on the sign of $y$ and is hence independent of the mesh size (for the considered values of $h$). Turning to the behavior with respect to $\gamma$, we first note that the Newton iteration failed to converge for larger values. This can be explained by the fact that for larger values of $\gamma$, the critical set where $y_i\in [-\gamma,0]$ and hence $y_i+\gamma\chi_i\in [0,\gamma]$ becomes larger so that we expect the local convergence of Newton's method to become an issue. For smaller values of $\gamma$, the Newton iteration converges quickly, and the approximation of $y$ and $p$ is independent of $\gamma$. This is not the case for $\chi$, where the approximation becomes worse. Here we point out that while \eqref{eq:BKKTFE3prime} is equivalent to \eqref{eq:BKKTFE3} for any value of $\gamma$, this only holds for exact solutions. A simple pointwise inspection shows that if \eqref{eq:BKKTFE3prime} does not hold exactly but only up to a residual of $\varepsilon$, then $\chi_\varepsilon = \chi + \mathcal{O}(\frac\varepsilon\gamma)$. Finally, we see that the Newton method is robust with respect to $\alpha$, and the approximation of $y$ and $p$ even improves for smaller $\alpha$. However, this seems to be a particularity of this example, since the Tikhonov parameter enters the data through the construction of this exact solution.

\bigskip

For the second example, we choose
\begin{equation*}
    y(x_1,x_2) = p(x_1, x_2)
    = \begin{cases}
        \big((x_1 - \frac{1}{2})^4 + \frac{1}{2}(x_1 - \frac{1}{2})^3\big) \sin(\pi\,x_2) , 
        & x_1 < \frac{1}{2},\\
        0, & x_1 \geq \frac{1}{2}.
    \end{cases}
\end{equation*}
Note that $y$ and $p$ are twice continuously differentiable and vanish on the 
right half of the unit square. Therefore, the non-smoothness of the $\max$-function occurs on a set of positive measure in this example. 
Moreover, as $y$ and $p$ vanish at the same time, $\chi$ is not unique in this set. 
This example can thus be seen as a worst-case scenario. Nevertheless, our algorithm 
is able to produce reasonable results as \cref{tab:ex2} demonstrates. 
(Note that it does not make sense to list the discrete $L^\infty$ error for $\chi_h$, 
since $\chi$ is not unique as mentioned above. In contrast, we can now report relative errors for $p_h$.)

\begin{table}
    \caption{Numerical results in the second example}\label{tab:ex2}
    \centering
    \begin{tabular}{cccccc}
        \toprule
        $h$ & $\alpha$ & $\gamma$ & $\frac{\|y_h - y\|_{L^2}}{\|y\|_{L^2}}$
            & $\frac{\|p_h - p\|_{L^2}}{\|p\|_{L^2}}$
            & \#~Newton\\
        \midrule
        \num{0.030303030303030} & \num{0.0001} & \num{0.000000000001}
                                & \num{0.870788577013324} & 
        \num{0.016062214301974} 
        & 4 \\ 
        \num{0.015384615384615} & \num{0.0001} & \num{0.000000000001}
                                & \num{0.228093558993323} & 
        \num{0.004541351248362} 
        & 5 \\ 
        \num{0.007751937984496} & \num{0.0001} & \num{0.000000000001}
                                & \num{0.058208952886468} & 
        \num{0.001208716007171} 
        & 3 \\ 
        \num{0.003891050583658} & \num{0.0001} & \num{0.000000000001}
                                & \num{0.014691669074202} &  
        \num{0.0003118564884766301} 
        & 3 \\ 
        \midrule
        \num{0.007751937984496} & \num{0.0001} & \num{0.000001}
                                & -- & -- 
                                & no conv. \\ 
        \num{0.007751937984496} & \num{0.0001} & \num{0.00000001}
                                & -- & -- 
                                & no conv. \\ 
        \num{0.007751937984496} & \num{0.0001} & \num{0.0000000001}
                                & \num{0.058208877119612} & 
        \num{0.001208715657339}
        & 3 \\ 
        \num{0.007751937984496} & \num{0.0001} & \num{0.00000000000001}
                                & \num{0.058208952886455} & 
        \num{0.001208716007172}
        & 3 \\ 
        \midrule
        \num{0.007751937984496} & \num{0.01} & \num{0.000000000001}
                                & \num{0.003006983383018} & 
        \num{0.001747429748373}
        & 2 \\ 
        \num{0.007751937984496} & \num{0.001} & \num{0.000000000001}
                                & \num{0.016592135633633} & 
        \num{0.001511691646965} 
        & 2 \\ 
        \num{0.007751937984496} & \num{0.00001} & \num{0.000000000001}
                                & \num{0.169164037555559} & 
        \num{0.0008659033532584455} 
        & 5 \\ 
        \num{0.007751937984496} & \num{0.000001} & \num{0.000000000001}
                                & -- & -- 
                                & no conv. \\ 
        \bottomrule
    \end{tabular}
\end{table}

The algorithm shows a similar behavior as in the first example.
Again, we observe quadratic convergence with respect to mesh refinement and that the Newton method does not converge if $\gamma$ is chosen too large.
(Note that in this example, the optimal state $y$ is scaled differently, which influences the effect of $\gamma$ in \eqref{eq:BKKTFE3prime}.)
In contrast to the first example, smaller values of $\alpha$ lead to worse numerical approximation and slower convergence of the Newton method (and even non-convergence for $\alpha = 10^{-6}$).
This is a typical observation which is also made in case of smooth optimal control problems.

\bigskip

In summary, one can conclude that our semi-smooth Newton-type method 
seems to be able to solve the discrete optimality system \eqref{eq:BKKTFE} for a certain range of parameters, even in genuinely non-smooth cases.
However, a comprehensive convergence analysis is still lacking,
and the choice of the parameter $\gamma$ appears to be a delicate issue.
Moreover, as already mentioned in Remark~\ref{rem:strongstatnum}, it is completely unclear 
how to incorporate the sign condition in \eqref{eq:SKKTpsign} into the algorithmic framework. 
This is the subject of future research.

\appendix
\section{Smooth ``characteristic'' functions of open sets}

\begin{lemma}\label{lem:randomlemma}
    For every open set $D \subseteq \R^d$  
    there exists a function $\psi \in C^\infty(\R^d)$ 
    such that $\psi > 0$ everywhere in $D$ and $\psi \equiv 0$ in $\R^d \setminus D$.
\end{lemma}

\begin{proof}
    Since the collection of all open balls with rational radii and rational centers 
    forms a base of the Euclidean topology on $\R^d$, 
    given an arbitrary but fixed open set $D$ we find (non-empty) open balls 
    $B_n \subseteq D$, $n \in \mathbb{N}$, such that 
    \begin{equation*}
        D = \bigcup_{n=1}^\infty B_n.
    \end{equation*}
    For every ball $B_n$, there is a smooth rotational symmetric bump function 
    $\psi_n \in C^\infty(\R^d)$ with $\psi_n > 0$ in $B_n$ 
    and $\psi_n \equiv 0 $ in $\R^d \setminus B_n$. Defining
    \begin{equation*}
        \psi := \sum_{n=1}^\infty \frac{\psi_n}{2^n \|\psi_n\|_{H^n(\R^d)}}.
    \end{equation*}
    it holds that $\psi > 0$ in $D$, $\psi \equiv 0$ in $\R^d \setminus D$, 
    and $\psi \in H^n(\R^d)$ for all $n\in \N$. Sobolev's embedding theorem then yields the claim.
\end{proof}

\section{Approximation of functions in \texorpdfstring{$\scriptstyle L^\infty(\Omega;\{0,1\})$}{L∞(Ω;{0,1})} and \texorpdfstring{$\scriptstyle L^\infty(\Omega;[0,1])$}{L∞(Ω;[0,1])}}
\label{sec:approx}

\begin{lemma}\label{MassEffect}
    If $(X, \Sigma,\mu)$ is a finite measure space and if $\mathcal{A}:=\{A_i\}_{i \in I}$, 
    $I \subset \R$, is a collection of measurable disjoint sets $A_i \in \Sigma$,
    then there exists a countable set $N \subset I$ such that $\mu(A_i) = 0$ 
    for all $i \in I \setminus N$.
\end{lemma}

\begin{proof}
    Define $\mathcal{A}_k := \{ A_i \in \mathcal{A} : \mu(A_i)  \geq 1/k\}$, $k \in \mathbb{N}$. 
    Then every $A_i \in \mathcal{A}$ with $\mu(A_i) > 0$ is contained in at least one $\mathcal{A}_k$.  
    Suppose $A_{i_1},\dots, A_{i_m}$ are contained in $\mathcal{A}_k$. 
    Then their disjointness implies
    \begin{equation*}
        \frac{m}{k} \leq \sum_{l=1}^m \mu(A_{i_l}) \leq \mu(X),
    \end{equation*}  
    and thus $m \leq k \mu(X) < \infty$. Consequently, 
    there can only be finitely many $A_i$ in each $\mathcal{A}_k$ 
    such that the set $\{i \in I : \mu(A_i) > 0\}$ is countable.
\end{proof}

\begin{lemma}\label{lem:approx}\ 
    \begin{enumerate}[label=(\roman*)]
        \item\label{it:approxopen}
            If $A\subseteq \Omega$ is open, then there exists a sequence 
            $\varphi_n \in Y$ with $\lambda^d(\{\varphi_n = 0\})=0$ such that 
            \begin{equation*}
                \mathbb{1}_{\{\varphi_n > 0\}} \to \mathbb{1}_{A}\text{ pointwise a.e. in }\Omega.
            \end{equation*}
        \item\label{it:approxborel}
            If $A\subseteq \Omega$ is Lebesgue measurable, 
            then there exists a sequence of open sets $A_n \subseteq \Omega$ such that
            \begin{equation*}
                \mathbb{1}_{A_n} \to  \mathbb{1}_{A}\text{ pointwise a.e. in }\Omega. 
            \end{equation*}
        \item\label{it:approxborelweight}
            If $A\subseteq \Omega$ is Lebesgue measurable and if $c \in (0, 1]$ is arbitrary but fixed, 
            then there exists a sequence of Lebesgue measurable sets $A_n \subseteq A$ such that
            \begin{equation*}
                \mathbb{1}_{A_n} \weakly^{*}   c\mathbb{1}_{A}\text{ in } L^\infty(\Omega). 
            \end{equation*}
        \item\label{it:approxsimple}
            If $\chi : \Omega \to \R$ is a simple function of the form 
            \begin{equation*}
                \chi := \sum_{k=1}^K c_k \mathbb{1}_{B_k}
            \end{equation*}
            with $K\in\N$, $c_k \in (0, 1]$,
            and $B_k \subseteq \Omega$ Lebesgue measurable and mutually disjoint, 
            then there exists a sequence of Lebesgue measurable sets $A_n \subseteq \Omega$ such that
            \begin{equation*}
                \mathbb{1}_{A_n} \weakly^{*} \chi \text{ in }L^\infty(\Omega).
            \end{equation*}
    \end{enumerate}
\end{lemma}

\begin{proof}
    Ad \ref{it:approxopen}: 
    Let $A \subseteq \Omega$ be an arbitrary but fixed open set. By \cref{lem:randomlemma}, 
    there exist functions $\psi, \phi \in C^\infty(\R^d)$ with
    \begin{equation*}
        \psi > 0 \text{ in } \Omega, \quad \psi \equiv 0 \text{ in } \R^d \setminus \Omega, 
        \quad \text{and} \quad
        \phi > 0 \text{ in } A, \quad \phi \equiv 0 \text{ in } \R^d \setminus A.
    \end{equation*}
    So, if we define $\varphi_{\varepsilon} := \phi - \varepsilon \psi$, 
    then $\varphi_{\varepsilon} \in Y \cap C(\overline{\Omega})$ 
    for all $\varepsilon \in (0, 1)$ and $\varphi_{\varepsilon}(x) \to \phi(x)$ 
    for all $x \in \Omega$ as $\varepsilon \to 0$ 
    (here and in the following, we always use continuous representatives). 
    Moreover, the sign conditions on $\psi$ and $\phi$ imply that
    \begin{equation}\label{eq:charcatconv}
        \mathbb{1}_{\{\varphi_\varepsilon > 0\}} \to \mathbb{1}_{A}\text{ pointwise a.e. in }\Omega.
    \end{equation}
    Consider now some $\varepsilon_1, \varepsilon_2 \in (0, 1)$ 
    with $\varepsilon_1 \neq \varepsilon_2$. Then it holds that
    \begin{equation*}
        \begin{multlined}[c][0.9\displaywidth]
            \{x \in \Omega : \varphi_{\varepsilon_1}(x)  = 0 \} 
            \cap \{x \in \Omega : \varphi_{\varepsilon_2}(x)  = 0 \} \\[1ex]
            = \{x \in \Omega : \phi(x) - \varepsilon_1 \psi(x)   = 0 
            \text{ and } \phi(x) - \varepsilon_2 \psi(x) = 0\}  
            = \{x \in \Omega : \varepsilon_1  = \varepsilon_2  \} = \emptyset,
        \end{multlined}
    \end{equation*}
    showing that the collection 
    $( \{ x \in \Omega :  \varphi_{\varepsilon}(x) = 0\} )_{\varepsilon \in (0, 1)} $ is disjoint.   
    Analogously to the proof of \cref{lem:preliminarycharacterization}, 
    we now obtain by means of \cref{MassEffect} that there exists a sequence $(\varepsilon_n)$ 
    with $\varepsilon_n \to 0$ as $n \to \infty$ such that
    \begin{equation*}
        \lambda^d(\{ x \in \Omega :  \varphi_{\varepsilon_n}(x) = 0\}) = 0\quad \forall n \in \mathbb{N}.
    \end{equation*}
    Together with \eqref{eq:charcatconv}, this establishes the assertion of part \ref{it:approxopen}.

    Ad \ref{it:approxborel}: The outer regularity of the Lebesgue measure implies the existence
    of open sets $\tilde{A}_n \subseteq \Omega$ such that
    \begin{equation*}
        A \subseteq \tilde{A}_n 
        \quad \text{and} \quad 
        \lambda^d(\tilde{A}_n \setminus A) < \frac{1}{n}\quad \forall n \in \mathbb{N}.
    \end{equation*}
    Let us define
    \begin{equation*}
        A_n := \bigcap_{m=1}^n \tilde{A}_m \quad \forall\, n \in \mathbb{N}.
    \end{equation*}
    Then $A_n$ is open for all $n \in \mathbb{N}$, and it holds that
    \begin{equation*}
        A_{n+1}\subseteq A_n,\quad A \subseteq A_n, 
        \quad \text{and} \quad \lambda^d(A_n \setminus A) < \frac{1}{n}\quad \forall n \in \mathbb{N}. 
    \end{equation*}
    The above implies that 
    \begin{equation*}
        \mathbb{1}_{A_n}(x) \to  \mathbb{1}_{A}(x)
        \quad \forall \, x \in A \cup \bigcup_{n \in \mathbb{N}}  \Omega \setminus A_n 
        = \Omega \setminus \left ( \bigcap_{n\in \mathbb{N}} A_n \setminus A \right ). 
    \end{equation*}
    Since the exceptional set appearing above has measure zero, this proves \ref{it:approxborel}.

    Ad \ref{it:approxborelweight}: We apply a homogenization argument.
    Given $n \in \mathbb{N}$, let us define 
    \begin{equation*}
        B_{n} := \bigcup_{k \in \mathbb{Z}^d} \frac{1}{n} k 
        + \left [0, \frac{1}{n} \right ]^{d-1} \times \left [0, \frac{1}{n} c \right ] 
        \subseteq \R^d.
    \end{equation*}
    The sequence $(\mathbb{1}_{B_n})$ is bounded in $L^\infty(\R^d)$, 
    and we may extract a subsequence (not relabeled) such that 
    $\mathbb{1}_{B_n} \weakly^{*}  \chi$ in $L^\infty(\R^d)$ 
    for some $\chi \in L^\infty(\R^d)$. 
    Consider now an arbitrary but fixed $\varphi \in C_c^\infty(\R^d)$. Then it holds that
    \begin{equation*}
        \begin{aligned}
            \int_{\R^d} \mathbb{1}_{B_n} \varphi\, \mathrm{d}x
            & = \int_{\R^d} c  \varphi  \, \mathrm{d}x  + c\sum_{k \in \mathbb{Z}^d}  
            \int_{\frac{1}{n} k + \left [0, \frac{1}{n} \right ]^{d} } \varphi\left (\frac{1}{n} k \right ) 
            - \varphi(x) \, \mathrm{d}x\\
            \MoveEqLeft[-5]  +  \sum_{k \in \mathbb{Z}^d} \int_{\frac{1}{n} k 
            + \left [0, \frac{1}{n} \right ]^{d-1} \times \left [0, \frac{1}{n} c \right ]} 
            \varphi(x) -  \varphi\left (\frac{1}{n} k \right )\mathrm{d}x \\
            &\to \int_{\R^d} c  \varphi\, \mathrm{d}x.
        \end{aligned}
    \end{equation*}
    Using standard density arguments and the uniqueness of the weak-$\ast$ limit, 
    we deduce from the above that 
    $\mathbb{1}_{B_n} \weakly^{*}  \chi \equiv c$ in $L^\infty(\R^d)$ 
    for the whole original sequence $(\mathbb{1}_{B_n})_{n \in \mathbb{N}}$. 
    But now for all $v \in L^1(\Omega)$, it holds that
    \begin{equation*}
        \int_\Omega \mathbb{1}_{B_n \cap A}  v  \,\mathrm{d}x 
        = \int_{\R^d}\mathbb{1}_{B_n}  \mathbb{1}_{A} v  \,\mathrm{d}x 
        \to \int_\Omega c \mathbb{1}_{A} v \,\mathrm{d}x, 
    \end{equation*}
    i.e., $\mathbb{1}_{B_n \cap A} \weakly^{*}   c\mathbb{1}_{A} $ 
    in $L^\infty(\Omega)$. This gives the claim with $A_n := B_n \cap A$.

    Ad \ref{it:approxsimple}:
    According to part \ref{it:approxborelweight}, we can find
    for every $k \in \{1,\dots,K\}$ a sequence of Lebesgue measurable sets $A_{k,n}\subseteq B_k$ such that
    \begin{equation*}
        \mathbb{1}_{A_{k, n}} \weakly^{*}  c_k\mathbb{1}_{B_k}
        \text{ in } L^\infty(\Omega) \text{ as }n \to \infty. 
    \end{equation*}
    Defining $A_n := \bigcup_{k=1}^K A_{n, k}$, the claim follows immediately by superposition. 
\end{proof}

\section*{Acknowledgments}
C.~Clason was supported by the DFG under grant CL\,487/2-1, and
C.~Christof and C.~Meyer were supported by the DFG under grant ME\,3281/7-1, all within the priority programme SPP 1962 ``Non-smooth and Complementarity-based Distributed Parameter Systems: Simulation and Hierarchical Optimization''.

\printbibliography

\end{document}